\documentclass[12pt]{amsart}
\usepackage{amsfonts}
\usepackage{amsmath}
\usepackage{amsthm}
\usepackage{amssymb}
\usepackage{latexsym}
\usepackage{graphicx}
\usepackage{epsfig}
\usepackage{url}
\usepackage{ulem}
\usepackage{multicol}
\makeindex
\usepackage{enumerate}
\usepackage{tabularx}
\usepackage{enumerate,hyperref,mathrsfs}
\usepackage{comment}

\def\beq{\begin{equation}}
\def\eeq{\end{equation}}

\def\ben{\begin{enumerate}}
\def\een{\end{enumerate}}

\def\RR{\mathbb{R}}
\def\CC{\mathbb{C}}
\def\NN{\mathbb{N}}

\def\HH{\mathbb{H}}
\def\ax{\langle x \rangle}
\def\axs{\langle x, x^{\ast} \rangle}

\def\cB{ {\mathcal B} }
\def\cC{ {\mathcal C} }

\def\cG{ {\mathcal G} }
\def\cH{ {\mathcal H} }
\def\cI{ {\mathcal I} }

\def\cM{{\mathcal M}}

\def\cZ{ {\mathcal Z} }

\def\azs{\langle z, z^* \rangle}

\newcommand{\setmult}[2]{#1 #2}
\newcommand{\leftact}[2]{#1 \cdot #2}
\newcommand{\lead}[1]{\operatorname{lm}(#1)}
\newcommand{\nonlead}[1]{\operatorname{Nlm}(#1)}
\newcommand{\Id}[1]{1_{#1}}
\def\finite{\mbox{\tiny finite}}

\def\real{\rm rr}

\def\rC{\mathfrak{C}}

\newtheorem{theorem}{Theorem}[section]
\newtheorem{lemma}[theorem]{Lemma}
\newtheorem{thm}[theorem]{Theorem}
\newtheorem{prop}[theorem]{Proposition}

\theoremstyle{definition}

\newtheorem{exa}[theorem]{Example}

\newtheorem{corollary}[theorem]{Corollary}

\def\subset{\subseteq}
\def\supset{\supseteq}

\newcommand{\rr}[1]{\sqrt[\real]{#1}}

\renewcommand{\tilde}{\widetilde}

\numberwithin{equation}{section}

\begin{document}

\title[Non-Commutative Nullstellensatz]{A Real Nullstellensatz for Matrices of 
Non-Commutative Polynomials }

\author{Christopher S. Nelson}



\begin{abstract}
This article extends the classical Real Nullstellensatz
to matrices of polynomials in a free $\ast$-algebra
$\RR\axs$ with $x=(x_1, \ldots, x_n)$.  This result is a generalization of a
result of Cimpri\v c, Helton, McCullough, and the author.

In the free left $\RR\axs$-module $\RR^{1 \times \ell}\axs$
we introduce notions of the
(noncommutative) zero set of a
left $\RR\axs$-submodule and of a real left $\RR\axs$-submodule. We prove that
every element from $\RR^{1 \times \ell}\axs$ whose zero set contains the
intersection of zero sets of
elements from a finite subset $S \subset \RR^{1 \times \ell}\axs$ belongs to the
smallest
real left $\RR\axs$-submodule containing $S$.  Using this, we derive a
nullstellensatz for matrices of polynomials in $\RR\axs$.

The other main contribution of this article is an efficient, implementable
algorithm which for every finite subset $S \subset \RR^{1 \times \ell}\axs$
computes the smallest real left $\RR\axs$-submodule containing $S$.  This
algorithm terminates in a finite number of steps.  By taking advantage of the
rigid structure of $\RR\axs$, the
algorithm presented here is an improvement upon the previously known
algorithm for $\RR\axs$.
\end{abstract}

\maketitle

\newpage

\section{Introduction}
\label{sect:intro}

 This article establishes a non-commutative analog
  of the classical real Nullstellensatz.
  The results of Krivine \cite{Kri1}, \cite{Kri2},
Dubois \cite{D69}, and Risler \cite{R70}
  established a real nullstellensatz in classical (commutative) real algebraic
geometry.
For a modern survey of real algebraic geometry (RAG), see the survey by
Scheiderer
\cite{sche} the book of Marshall \cite{mm3}, and the book of Bochnak, Coste and
Roy \cite{bcr}.

The main aim of this paper is to extend the real nullstellensatz to
non-commutative algebras, in particular free algebras.
The earliest such result was proved by George Bergman in an algebra without
involution, settling a conjecture of Helton and McCullough \cite{HM04}.
Unfortunately, in an algebra with involution, \cite{HM04} contains a
counterexample to extending Bergman's result. 
In \cite{HMP} a major case is settled
in a free $\ast$-algebra,
but much remained open.

A main thrust of real algebraic geometry, in addition to the nullstellensatz,
is the positivstellensatz.  This holds as well for non-commutative algebras,
and is an active area of research dating back to ideas of Putinar \cite{P} and
Helton and McCullough \cite{HM04}. There is a convex branch
of that subject, see \cite{HM12}, \cite{HKM12}, and \cite{KS}, and this
article feeds into that thereby underlying the sequel \cite{HKN13} to this
paper.

  This introduction is arranged as follows: in $\S$
\ref{sub:notation} some basic notation will be introduced; in $\S$
\ref{sub:commutInsp} we will discuss the commutative inspiration for the
non-commutative nullstellensatz; in $\S$  \ref{sub:NCpolys} we will lay out
basic definitions for non-commutative polynomials; in $\S$ \ref{sub:leftMods}
we discuss non-commutative analogs of zero sets and radicals; the main results
of the article are then presented in $\S$ \ref{sub:overview}; and finally an
outline of the article is given in $\S$ \ref{sub:guide}.

  Our approach to Noncommutative Real Algebraic Geometry is motivated
  by \cite{HMP}; for alternative approaches see \cite{sch2} and \cite{mm}.

\subsection{Notation}
\label{sub:notation}

Given positive integers $\nu$ and $\ell$,
let $\RR^{\nu \times \ell}$ denote the space of $\nu
\times \ell$ real matrices.
Let $E_{ij} \in \RR^{\nu \times \ell}$
\index{Eij@$E_{ij}$} denote
the matrix with a $1$ as the $ij^{th}$ entry and a $0$ for all other entries.
Let $e_j \in \RR^{1 \times \ell}$ \index{ej@$e_j$} denote
the
row vector with $1$ as
the $j^{th}$ entry and a $0$ as all other entries.
Let $\Id{\nu} \in \RR^{\nu \times \nu}$ denote the $\nu \times \nu$
identity matrix.
Let $A^* \in \RR^{\ell \times \nu}$ denote the transpose of a matrix $A \in
\RR^{\nu \times \ell}$.
Let $\mathbb{S}^{k} \subset \RR^{k \times k}$\index{S^k@$\mathbb{S}^k$} denote
the space of real symmetric $k \times k$ matrices.

Although our notation does not correspond to them, one who wishes a general
orientation to non-commutative algebras can see Goodearl \cite{Goo}, and for
free algebras see P. M. Cohn \cite{Coh}.

\subsection{Commutative Inspiration}
\label{sub:commutInsp}

If $I$ is an ideal in the set of commutative polynomials $\RR[x]$ with real
coefficients, we say $I$ is {\bf real} if whenever there are some
polynomials $p_i \in \RR[x]$ which satisfy
\[
 \sum_i^{\rm finite} p_i^2 \in I
\]
then each $p_i \in I$.
The real Nullstellensatz \cite{D69},
\cite{R70} states that if $q \in \RR[x]$, then $q(a) = 0$ for all tuples of
real scalars $a$ such that $p(a) = 0$ for each $p \in I$ if and only if $q$ is
in the {\bf real radical of $I$}, that is, the smallest real ideal containing
$I$.

A more recent result of Cimpri\v c \cite{Cim13} extends this result to matrices
of commutative polynomials.  A left submodule $I$ of the free
$\RR[x]$-module $\RR[x]^{1 \times \ell}$ is {\bf
real} if whenever $p_i \in \RR[x]^{1 \times \ell}$ satisfy
\[
 \sum_i^{\rm finite} p_i \otimes p_i \in \RR[x]^{\ell \times 1} \otimes I +
I \otimes \RR[x]^{1 \times \ell}
\]
then each $p_i \in I$.  On $\RR[x]^{1 \times \ell}$, Cimpri\v c's result states
that if $q
\in \RR[x]^{1 \times \ell}$, then $q(a) = 0$ for all tuples of
real scalars $a$ such that $p(a) = 0$ for each $p \in I$ if and only if $q$ is
in the {\bf real radical of $I$}, that is, the smallest real left submodule
containing
$I$.

\subsection{Non-Commutative Polynomials}
\label{sub:NCpolys}

We now turn our attention to the space of non-commutative polynomials.

Let $\axs$\index{$\axs$} denote the monoid freely generated by $x = (x_1,
\ldots, x_g)$ and $x^* = (x_1^*, \ldots, x_g^*)$---that is, $\axs$ consists of
words in the $2g$ free letters $x_1, \ldots, x_g, x_1^*, \ldots, x_g^*$,
including the empty word $\emptyset$, which plays the role of the identity $1$.
Let  $\RR\axs$\index{Rxxs@$\RR\axs$} denote the
$\RR$-algebra freely generated by $\axs$, i.e., the
elements of $\RR\axs$
are polynomials in the non-commuting variables $\axs$
with coefficients
in $\RR$.  Call elements of $\RR\axs$
\textbf{non-commutative}\index{non-commutative (NC) polynomials} or \textbf{NC}
polynomials.

The \textbf{involution}\index{involution} on
$\RR\axs$ is defined linearly so that $(x_i^*)^* = x_i$ for each variable $x_i$
and
$(pq)^* = q^*p^*$ for each $p, q \in \RR\axs$.
For example,
\[
 \left( x_1x_2x_3 + 2x_3^*x_1 - x_3\right)^* = x_3^*x_2^*x_1^* + 2x_1^*x_3 -
x_3^*
\]

\subsubsection{Evaluation of NC Polynomials}

NC polynomials can be evaluated at a tuple of matrices in a natural way.
 Let $X = (X_1, \ldots, X_g) \in \left( \RR^{n \times n} \right)^g$. Given $p
\in
\RR\axs$, let $p(X)$ denote the matrix defined by
replacing each $x_i$ in $p$ with $X_i$, each $x_i^*$ in $p$ with $X_i^{\ast}$,
and replacing the empty word with $\Id{n}$.  Note that $p^*(X) =
p(X)^*$ for all $p \in \RR\axs$.

For example, if
\[
p(x) = x_1^2 - 2x_1x_2^* -3, \quad
X_1 = \begin{pmatrix}
1&2\\
2&4
\end{pmatrix}
\quad
\mbox{and}
\quad
X_2 = \begin{pmatrix}
0&-1\\
1&-1
\end{pmatrix}
\]
then
\begin{align}
 \notag
p(X) &= X_1^2 - 2X_1 X_2^* - 3 (\Id{2}) \\
\notag
&=
\begin{pmatrix}
1&2\\
2&4
\end{pmatrix}\begin{pmatrix}
1&2\\
2&4
\end{pmatrix}
- 2 \begin{pmatrix}
1&2\\
2&4
\end{pmatrix}
\begin{pmatrix}
0&1\\
-1&-1
\end{pmatrix}
-
\begin{pmatrix}
3&0\\
0&3
\end{pmatrix}
\\
\notag
&
=
\left(
\begin{array}{cc}
 6 & 12 \\
 18 & 21 \\
\end{array}
\right)
\end{align}

\subsubsection{Matrices of NC Polynomials}

The space of $\nu \times \ell$ matrices with entries in $\RR\axs$
will be
denoted as $\RR^{\nu \times \ell}\axs$. Each $p \in \RR^{\nu \times \ell}\axs$
can be expressed as
\[
p = \sum_{w \in \axs} A_w \otimes w \in \RR^{\nu \times \ell}\otimes \RR\axs.
\]
Given a tuple $X$ of real $n \times n$ matrices, let $p(X)$ denote
\[
p(X) = \sum_{w \in \axs} A_w \otimes w(X) \in \RR^{\nu n \times \ell n}
\]
where $\otimes$ denotes the Kronecker product.
The involution on $\RR^{\nu \times \ell}\axs$ is
given by
\[
p^*= \left(\sum_{w \in \axs} A_w \otimes w\right)^* = \sum_{w \in \axs}
A_w^{\ast} \otimes w^* \in \RR^{\ell \times \nu}\axs.
\]
Note that $p^*(X) = p(X)^*$ for any tuple $X$.
If $p \in \RR^{\nu \times \nu}\axs$, we say $p$
is \textbf{symmetric}\index{symmetric
polynomial} if $p = p^*$.

\subsubsection{Degree of NC Polynomials}
Let $|w|$ denote the \textbf{length}\index{length of a word in $\axs$} of a word
$w \in
\axs$.
A \textbf{monomial}\index{monomial} in $\RR^{\nu \times \ell}\axs$ is a
polynomial
of the form
$E_{ij} \otimes m$, where $m \in \axs$.
Let $\cM^{\nu \times \ell}$ denote the set of monomials in $\RR^{\nu \times
\ell}\axs$
The {\bf length} or
 \textbf{degree}\index{degree} of a monomial $E_{ij} \otimes m$ is $|E_{ij}
\otimes m| := |m|$.

If $p$ is a NC polynomial, define the degree of $p$, denoted $\deg(p)$, to be
the
largest degree of any monomial appearing in $p$.
A NC polynomial $p$ is \textbf{homogeneous of
degree
$d$}\index{degree!homogeneous}
if every monomial appearing in $p$ has degree $d$.
If $W$ is a subspace of $\RR^{\nu \times \ell}\axs$, define $W_d$
\index{Wd for a vector
space W in Rxxs@ $W_d$ for a vector space $W \subset
\RR\axs$}\index{RRvlxxsd@$\RR^{\nu \times \ell}\axs_d$} to be
the space spanned by all elements of $W$ with degree at most $d$.

\subsubsection{Operations on Sets}

If $A, B \subset \RR^{\nu \times
\ell}\axs$, then define $A + B$ to be
\[
 A + B := \{ a + b \mid a \in A, b \in B \} \subset \RR^{\nu \times \ell}\axs.
\]
In the case that
$A \cap B = \{0\}$, we also denote $A + B$ as $A \oplus B$\index{$\oplus$}; the
expression $A \oplus B$ always asserts that $A \cap B =
\{0\}$.
If $A \subset \RR^{\nu \times \ell} \axs$ and $B
\subset \RR^{\ell \times \rho} \axs$,
let $\setmult{A}{B}$\index{$\setmult{A}{B}$}
be
\[
 AB := \operatorname{Span}( \{ab \mid a \in A, b \in B \}) \subset \RR^{\nu
\times \rho}\axs.
\]
If $A \subset \RR^{\nu \times \ell}\axs$, let
\[A^{\ast}  := \{a^* \mid a
\in A\} \subset
\RR^{\ell
\times \nu}
\axs.\]
If $A \subset \RR^{\nu \times \ell}$ and $B \subset \RR\axs$, then $A \otimes B$
is
\[
 A \otimes B := \operatorname{Span}(  \{ a \otimes b \mid a \in A, b \in B \} ).
\]

If $p \in \RR^{\nu \times \ell}\axs$, then expressions of the form $p + A$,
$pB$,
$Cp$, $D \otimes p$, where $A$, $B$, $C$, and $D$, are sets, denote $\{p\} +
A$, $\{p\}B$, $C\{p\}$, and $D \otimes \{p\}$ respectively.

\subsection{Left \texorpdfstring{$\RR\axs$}{[R x xs]}-Modules}
\label{sub:leftMods}
For $\RR\axs$, there is a ``Non-Commutative Left Real Nullstellensatz''. Let
$p_1, \ldots, p_k, q \in
\RR\axs$.  If $q(X)v = 0$ for every $(X,v) \in \bigcup_{n \in \NN} \left( \RR^{n
\times n}
\right)^g \times \RR^n$ such that $p_1(X)v = \cdots = p_k(X)v = 0$, then $q$ is
an element
of the ``real radical'' of the left ideal generated by $p_1, \ldots, p_k$
\cite{chmn}.  To generalize this result to $\RR^{\nu \times \ell}\axs$, we now
generalize the notion of left ideal and real left ideal to non-square
matrices of NC polynomials.

The space
$\RR^{1 \times \ell}\axs$ is a
free left $\RR\axs$-module.
That is,
if $q \in \RR\axs$, $A  \in \RR^{1 \times \ell}$ and $r \in \RR\axs$, then
\[
\leftact{q}{(A \otimes r)} :=  (\Id{\nu} \otimes q)(A \otimes r)=
A \otimes qr.
\]
In the sequel, we will simplify notation by identifying $q$ with
$\Id{\nu}
\otimes q$ and
simply writing $q(A \otimes r)$ when we mean $\leftact{q}{(A \otimes r)}$.
We will also simplify our terminology by referring to left
$\RR\axs$-submodules $I \subset \RR^{1 \times \ell}\axs$ as
 {\bf left modules}.

\subsubsection{Real Left Modules}

Let $I \subset \RR^{1 \times \ell}\axs$ be a left module.
We say that $I$ is \textbf{real}\index{real!left module}
if whenever
\[ \sum_i^{\finite} p_i^*p_i \in \setmult{\RR^{\ell \times 1}}{I}
+
\setmult{I^{\ast}}{\RR^{1 \times \ell}} \]
for some $p_i \in \RR^{1 \times \ell}\axs$, then each $p_i \in I$.
Note that $\RR^{\ell \times 1}I$ is the subspace of $\ell \times \ell$
matrices whose rows are elements of $I$, and $(\RR^{\ell
\times 1}I)^* = I^*\RR^{1 \times \ell}$ is the subspace of $\ell \times \ell$
matrices whose columns are elements of $I^*$. 

The following result shows that defining a real left module in terms of only $1
\times \ell$ matrices actually covers $\nu \times \ell$ matrices for any
dimension $\nu$.

\begin{prop}
\label{prop:diffMiReal}
 A left module $I \subset \RR^{1 \times \ell}\axs$
is real if and only if
whenever
\begin{equation}
 \label{eq:diffMi}
\sum_i^{\finite} p_i^*p_i \in \setmult{\RR^{\ell \times 1}}{I} +
\setmult{I^{\ast}}{\RR^{1 \times \ell}},
\end{equation}
for some $p_i \in
\RR^{\nu_i
\times \ell}\axs$ and some $\nu_i \in \NN$, then
each $p_i \in \setmult{\RR^{\nu_i \times 1}}{I}$.
\end{prop}

\begin{proof}
One direction is clear. For the converse, suppose $I$ is real, and suppose that
(\ref{eq:diffMi}) holds for some polynomials $p_i \in \RR^{\nu_i
\times \ell}\axs$.
For each $p_i$,
\[ p_i^*p_i =p_i^*\Id{\nu_i}p_i = \sum_{j=1}^{\nu_i} p_i^*  E_{jj
} p _ i =
\sum_{j=1}^{\nu_i}
(e_j^*p_i)^*(e_j^*p_i),\]
so that
\[ \sum_{i}^{\finite} p_i^*p_i = \sum_{i}^{\finite} \sum_{j=1}^{\nu_i}
(e_j^*p_i)^*(e_j^*p_i) \in \setmult{\RR^{\ell \times 1}}{I} +
\setmult{I^*}{\RR^{1 \times \ell}}.\]
Since $I$ is real, each $e_j^*p_i
\in I$.
Therefore, for each $i$,
$$p_i = \Id{\nu_i}p_i = \sum_{j=1}^{\nu_i} e_je_j^*p_i \in
\setmult{\RR^{\nu_i \times
1}}{I}. $$
\end{proof}

\subsubsection{The Real Radical}

An intersection of real left modules is itself a real left module.
Define the
\textbf{real radical}\index{real!radical} of a left module $I \subset \RR^{1
\times \ell}\axs$ to be
\[\rr{I} = \bigcap_{\substack{J \supseteq I,\\ J\ \mbox{\tiny real}}} J
=\text{the smallest real left module containing } I.\]

\subsubsection{Zero Sets of Left \texorpdfstring{$\RR\axs$}{[R x xs]}-Modules}

If $S \subset \RR^{1 \times \ell}\axs$, for each $n \in \NN$,
 define $V(S)^{(n)}$
to be
\[ V(S)^{(n)} := \{ (X,v) \in (\RR^{n \times n})^g \times \RR^{\ell n}
\mid p(X)v = 0\ \text{for every}\ p \in S\},\]
and define $V(S)$ to be
\[ V(S) := \bigcup_{n \in \NN} V(S)^{(n)}.\]
If $V \subset \bigcup_{n \in \NN} (\RR^{n \times n})^g \times
\RR^{\ell n}$, define $\cI(V)$ to be
\[ \cI(V) := \{ p \in \RR^{1 \times \ell}\axs \mid p(X)v = 0\ \text{for every}\
(X,v) \in V\}.\]
The set $\cI(V) \subset \RR^{1 \times \ell}\axs$ is clearly a left
module.
If $I \subset \RR^{1 \times \ell}\axs$ is a left module, define the
\textbf{(vanishing) radical}\index{radical!vanishing} of $I$ to be
\[\sqrt{I} := \cI(V(I)).\]
We say a left module is {\bf radical} if it is equal to its vanishing radical.

\begin{prop}
\label{prop:cIcCIsReal}
Let $V \subset \bigcup_{n \in
\NN} (\RR^{n \times n})^g \times
\RR^{\ell n}$. The space $\cI(V) \subset \RR^{1 \times \ell}\axs$ is a real left
module.
\end{prop}

\begin{proof}
 Suppose
\[\sum_i^{\finite} p_i^*p_i \in \RR^{\ell \times 1} \cI(V) +  \cI(V)^*\RR^{1
\times \ell},\]
where each $p_i \in \RR^{1 \times \ell}\axs$.
For each $(X,v) \in V$, we have
\[\sum_i^{\finite} p_i(X)^*p_i(X)v = 0  \quad \Longrightarrow
\quad \sum_i^{\finite}
v^*p_i(X)^*p_i(X)v = 0. \]
Therefore each $p_i(X)v = 0$, which implies that each $p_i \in \cI(V)$.
\end{proof}

Proposition \ref{prop:cIcCIsReal} implies that for each left module
$I \subset \RR^{1
\times \ell}\axs$,
\[ I \subset \rr{I} \subset \sqrt{I}.\]

\subsection{Main Results}
\label{sub:overview}

Here is the main result of this article, which is a generalization of
\cite[Theorem 1.6]{chmn} to the matrix case.
\begin{theorem}
\label{thm:mainFromNotes}
 Let $p_1, \ldots, p_k$ be such that each
$p_i \in \RR^{\nu_i \times \ell}\axs$ for some $\nu_i \in \NN$.
Define
\[
 J_\nu := \setmult{\RR^{\nu \times 1}}{\rr{\sum_{i=1}^k
\setmult{\RR^{1
\times \nu_i}\axs}{ p_i}}}
\]
for $\nu \in \NN$.
Let
$q \in \RR^{\nu \times \ell}\axs$.
Then $q(X)v = 0$ for all $(X,v) \in  \bigcup_{n \in \NN} (\RR^{n \times n})^g
\times
\RR^{\ell n}$ such that
$p_1(X)v, \ldots, p_k(X)v = 0$ if and only if  $q \in J_\nu$.

Consequently, if the left module
\begin{equation}
\label{eq:checkIReal}
\sum_{i=1}^k
\RR^{1
\times \nu_i}\axs p_i
\end{equation}
is real, and if $q(X)v = 0$
whenever $p_1(X)v, \ldots, p_k(X)v = 0$, then
$q$ is of the form
\[q = r_1p_1 + \cdots + r_k p_k,\]
where each $r_i \in \RR^{\nu \times \nu_i}\axs$.
\end{theorem}

This is proven in $\S$ \ref{sec:mainResults}.  An interesting corollary is
Corollary \ref{cor:extensionCH}, which gives the analog of this Theorem for
$\CC$ and $\HH$.  This Corollary follows from the observation (plus some
technical details) that $\CC$ and $\HH$ can be viewed as subspaces of $\RR^{2
\times 2}$ and $\RR^{4 \times 4}$ respectively.

In $\S$ \ref{subsub:rrAlg} we will present an algorithm for computing
$\rr{I}$ for a finitely-generated left module $I \subset \RR^{1 \times
\ell}\axs$.
This algorithm is a generalization of and an improvement upon the Real Algorithm
given in \cite{chmn}.
The following theorem, proven in $\S$
\ref{subsub:propsRrAlg}, states some of its appealing properties.

\begin{thm}
\label{thm:algorStops}
Let $I$ be the left module generated by
$\iota_1, \ldots, \iota_{\mu} \in \RR^{1 \times \ell}\axs$.  The following are
true for applying the algorithm described in
\S \ref{subsub:rrAlg} to $\iota_1, \ldots, \iota_{\mu}$.

\begin{enumerate}
\item
 \label{degreed}
If
$\deg(\iota_1), \ldots, \deg(\iota_{\mu}) \leq d$,
the polynomials involved in the algorithm all have degree
less than $2d$.

\item \label{stops}The algorithm is guaranteed to terminate in
a finite number of steps.

\item When the algorithm terminates, it outputs a reduced left
Gr\"obner basis for $\rr{I}$.
\end{enumerate}
\end{thm}

\subsection{Reader's Guide}
\label{sub:guide}

Sections \ref{sect:allAboutRC}, \ref{sec:COrders}, and \ref{sec:DoubleCOrders}
are technical sections which prove lemmas needed for the proof of the main
results.
Section \ref{sect:linFun} proves some important lemmas and closes with the proof
of Theorem \ref{thm:mainFromNotes}.
Section \ref{sec:CandH} proves an
extension of Theorem \ref{thm:mainFromNotes} to $\CC$ and $\HH$.
Section \ref{sec:realRadStuff}
proves a strong result, Theorem \ref{thm:reducedRealTest}, for verifying whether
a left module is real, which will be used for the Real Radical Algorithm. 
Section \ref{sect:LGB} is a technical section
discussing left Gr\"obner bases. Section \ref{sub:rralg} presents the Real 
Radical Algorithm
mentioned in Theorem \ref{thm:algorStops} and proves its nice properties.

\section{Right Chip Spaces and Factorization of Monomials}
\label{sect:allAboutRC}

We now introduce  a natural class of monomials
needed for the proofs, chip sets.
Further, the Real Radical Algorithm described in Theorem \ref{thm:algorStops}
makes extensive use of chip sets, which makes said algorithm very efficient.

Given monomials $m_1, m_2 \in \cM^{\nu \times \ell}$, we say that $m_2$ {\bf
right divides} $m_1$, or that $m_2$ is a {\bf right chip} of $m_1$, if $m_1 =
wm_2$ for some $w \in \axs$.  If $w \neq 1$, we say the division is {\bf
proper} or that $m_2$ is a {\bf proper right chip}.

\begin{exa}
 If
\[
m_1 = e_2 \otimes x_1x_2x_3 \quad \mbox{and} \quad m_2 = e_2 \otimes x_2x_3,
\]
then $x_1 m_2 = m_1$, so $m_2$ is a proper right chip of $m_1$.
\end{exa}

A {\bf right chip space} $\rC \subset \RR^{\nu \times \ell}\axs$ is a space
spanned by monomials $m$ such that if $m \in \rC$, then so are all of its right
chips.  A right chip space is {\bf finite} if it is finite dimensional.

The space of NC polynomials has a rigid structure which makes finding sums of
squares representations easy.  For example, Klep and Povh \cite{kp} showed that
to verify that a NC polynomial $p$ is a sum of squares, one needs only to use
the right chips of the terms of $p$.
In this section we prove some basic results about right chip spaces which will
be useful in proving the main results of this article.

\subsection{Constant Matrices in the Complement of a Full Right Chip Space}

The element $1 \in \RR\axs$ right divides any monomial in $\RR^{1 \times 1}\axs
=
\RR\axs$, hence $1 \in \rC$ for any right chip space $\rC
\subset \RR\axs$.  For dimensions $\ell > 1$, however, not all right chip
spaces contain all constants.
Define
\textbf{$\Gamma(\rC)$}\index{$\Gamma(\rC)$} to be
\[\Gamma(\rC) := \{ j \mid e_j \otimes 1 \in \rC\} \subset \{1, \ldots,
\ell\}.\]

\begin{lemma}
\label{lem:gammaC}
Let $\rC \subset \RR^{1 \times \ell}\axs$ be a right chip space.
Then $\RR\axs \rC$ is equal to
\[\RR\axs \rC = \bigoplus_{j \in \Gamma(\rC)} e_j \otimes \RR\axs.\]
\end{lemma}

\begin{proof}
If $e_j \otimes w \in \rC$ for some $w \in \axs$, then since $\rC$ is a full
right chip space,
$e_j \otimes 1 \in \rC$.  The result is clear from here.
\end{proof}

Of interest as well are spaces of the form $\rC^* \RR\axs \rC \subset
\RR^{\ell \times
\ell}\axs$.

\begin{lemma}
 \label{lemma:MsFFaxsM}
 Let $\rC \subset \RR^{1 \times \ell}\axs$ be a
right chip space.
Then $\rC^* \RR\axs \rC$ is equal to
\begin{equation*}
\rC^* \RR\axs \rC = \bigoplus_{i,j \in \Gamma(\rC)} E_{ij}
\otimes
\RR\axs.
\end{equation*}
\end{lemma}

\begin{proof}
 This is clear from Lemma \ref{lem:gammaC}.
\end{proof}

\begin{lemma}
\label{lem:noTheta}
 Let $\rC \subset \RR^{1 \times \ell}\axs$ be a right chip space,
and let $I \subset \RR^{1 \times \ell}\axs$ be a left module generated by
polynomials in $\RR\axs \rC$.  Then
\[(\RR^{\ell \times 1} I + I^* \RR^{1 \times \ell}) \cap \rC^*\RR\axs \rC
= \rC^* I + I^* \rC.\]
\end{lemma}

\begin{proof}
 Let $\Theta$ be the space defined by
\[\Theta = \bigoplus_{j \not\in \Gamma(\rC)} e_j \otimes \RR\axs.\]
By Lemma \ref{lem:gammaC}, $\RR^{1 \times \ell}\axs = \RR\axs \rC
\oplus \Theta$, and
\begin{align}
\notag
\rC^*\RR\axs \rC &= \bigoplus_{i,j \in \Gamma(\rC)} E_{ij} \otimes \RR\axs,&
\Theta^* \rC & = \bigoplus_{\substack{i \not\in \Gamma(\rC)\\ j \in
\Gamma(\rC)}} E_{ij} \otimes \RR\axs\\
\notag
 \rC^*\Theta &=
 \bigoplus_{\substack{i \in \Gamma(\rC)\\ j \not\in \Gamma(\rC)}} E_{ij} \otimes
\RR\axs,
 &\Theta^*\Theta&=\bigoplus_{i,j \not\in \Gamma(\rC)} E_{ij} \otimes \RR\axs.
\end{align}
Therefore,
\begin{equation}
\label{eq:oplusCTh}
 \RR^{\ell \times \ell}\axs = \rC^*\RR\axs \rC \oplus
\Theta^* \rC \oplus \rC^*\Theta \oplus
\Theta^*\Theta.
\end{equation}
Let $\iota_1, \ldots, \iota_i, \ldots \in \RR\axs \rC$ generate $I$.
Each $\iota \in \RR^{\ell \times 1} I + I^* \RR^{1 \times
\ell}$
is of the form
\[\iota = \sum_i^{\finite} (p_i^*\iota_i + \iota_i^*q_i),\]
for some $p_i, q_i \in \RR^{1 \times \ell}\axs $.
Decompose each $p_i$ as $\phi_{p_i} + \theta_{p_i}$ and each
$q_i$ as $\phi_{q_i} + \theta_{q_i}$ so that
$\phi_{p_i}, \phi_{q_i} \in \RR\axs \rC$ and $\theta_{p_i},
\theta_{q_i} \in \Theta$.  Then
\begin{align}
\label{eq:partInCsC}
\iota &=
\sum_i^{\finite}
(\phi_{p_i}^*\iota_i + \iota_i^*\phi_{q_i})
+\sum_i^{\finite} (\theta_{p_i}^*\iota_i) +\sum_i^{\finite}
(\iota_i^*{\theta_{q_i}}),
\end{align}
so that the first sum of (\ref{eq:partInCsC})
is in $\rC^*\RR\axs \rC$, the second in $\Theta^* \rC$, and
the third in $\rC^*\Theta$.  If $\iota \in \rC^*\RR\axs \rC$, by
(\ref{eq:oplusCTh}), $\iota$ is equal to the first sum in
(\ref{eq:partInCsC}), which is an element of $\rC^*I + I^*\rC$.
\end{proof}

\subsection{Unique Factorization of Monomials}

If $w \in \axs$ and $0 \leq d \leq |w|$,
one can factor $w$ uniquely as
$w = w_1w_2$, where $w_1, w_2 \in \axs$ with $|w_1| = d$ and $|w_2| = |w| - d$.
The following lemma generalizes this fact to $\RR^{\nu \times \ell}\axs$.

\begin{lemma}
\label{lem:canFactorMon}
Let $m = E_{ij} \otimes w \in \cM^{\nu \times \ell}\axs$ and $w \in \axs$.
For each $0 \leq d \leq |m|$, there exists a factorization of
$m$ as $m = m_1^*m_2$, where $m_1 \in \cM^{1 \times \nu}$, $m_2 \in \cM^{1
\times \ell}$,
$\deg(m_1) = d$, and $\deg(m_2) = |m| - d$.  Further, this factorization is
uniquely determined, up to scalar multiplication,  by
$m_1 = e_i \otimes w_1^*$, $m_2 = e_j \otimes w_2$, where $w = w_1w_2$, $w_1,
w_2 \in \axs$, with
$|w_1| = d$ and $|w_2| = |m|-d$.
\end{lemma}

\begin{proof}
It is clear that $m$ can be factored as $m = (e_i \otimes w_1^*)^*(e_j \otimes
w_2) = E_{ij} \otimes w$, where $w_1, w_2 \in \axs$ with $|w_1| = d$ and $|w_2|
= |m|-d$.
Conversely, suppose $m = m_1^*m_2$, where
\[m_1 = \sum_{\rho=1}^{\nu} \sum_{u \in \axs} A_{\rho,u} e_{\rho} \otimes u
\quad
\mbox{and}
\quad
m_2 =  \sum_{\sigma=1}^{\ell} \sum_{v \in \axs} B_{\sigma,v} e_{\sigma}
\otimes v,\]
for some $A_{\rho,u}, B_{\sigma,v} \in \RR$.
Then,
\begin{align}
\notag
m_1^*m_2 &= \sum_{\rho = 1}^{\nu} \sum_{\sigma=1}^{\ell} \sum_{u \in
\axs}\sum_{v \in \axs}
A_{\rho,u}B_{\sigma,v} E_{\rho\sigma} \otimes u^*v \\
\label{eq:rhoSigma}
 &= \sum_{\rho=1}^{\nu} \sum_{\sigma = 1}^{\ell} E_{\rho\sigma}
\otimes \left( \sum_{u \in \axs}
A_{\rho,u}  u\right)^*
\left(\sum_{v \in \axs} B_{\sigma,v} v\right)\\
\notag
&= E_{ij} \otimes w.
\end{align}
The terms of  (\ref{eq:rhoSigma}) with $\rho = i$ and $\sigma = j$ are equal to
$E_{ij}
\otimes w_1^*w_2$, which implies, by uniqueness of the factorization of $w$,
that
$A_{i,u} = B_{j,v} = 0$, for $u \neq w_1^*$
and $v \neq w_2$, and $A_{i,w_1^*}B_{j,w_2} = 1$.  The terms of
 (\ref{eq:rhoSigma}) with
 $\rho \neq i$ and $\sigma =
j$
are equal to $0$, which implies that each
$A_{\rho, u} = 0$ for $\rho \neq i$.  Similarly, each $B_{\sigma, v} = 0$ for
$\sigma \neq 0$.
Therefore $m_1 = A_{i,w_1^*}(e_i \otimes w_1^*)$ and $m_2 =
(1/A_{i,w_1^*}) (e_j \otimes w_2)$.
\end{proof}

Given a right chip space $\rC \subset \RR^{1 \times \ell}\axs$, there are some
special factorizations of monomials in $\cM^{1 \times \ell}$ and $\cM^{\ell
\times \ell}$ which will be useful.

\begin{lemma}
\label{lem:uniqueDecomp}
Let $\rC \subset \RR^{1 \times \ell}\axs$ be a right chip space.
Each monomial $m \in \RR\axs \rC$ has a unique
word $w
\in \axs$ of minimum length and a unique right chip $\bar{m} \in \rC$
such that $m = w \bar{m}$.
\end{lemma}

\begin{proof}
Each monomial in $\RR\axs \rC$ is of the form $w
\bar{m}$, with $w \in \axs$ and $\bar{m} \in
\rC$.  Uniqueness of the minimal $w \in \axs$
follows from Lemma \ref{lem:canFactorMon}.
\end{proof}

\begin{lemma}
\label{lem:repRRaxsrC}
 Let $\rC \subset \RR^{1 \times \ell}\axs$ be a right chip space.
A monomial $m \in \cM^{1 \times \ell}$ is in the set $\RR\axs \rC \setminus \rC$
if and only if it can be expressed as $m = w\bar{m}$, where $\bar{m} \in
\RR\axs_1 \rC \setminus \rC$ and $w \in \axs$.  Further, this representation is
unique.
\end{lemma}

\begin{proof}
Decompose $m \in \RR\axs \rC \setminus \rC$ as in Lemma \ref{lem:uniqueDecomp}
as $m = \hat{w} \hat{m}$, where $\hat{m} \in \rC$ and $\hat{w}$ is as small as
possible.
We cannot have $\hat{w} = 1$ since $m \not\in \rC$, so decompose $\hat{w}$ as
$\hat{w}_1\hat{w}_2$, where $|\hat{w}_2| = 1$.  Then $m = \hat{w}_1(\hat{w}_2
\hat{m})$, and by minimality
of $\hat{w}$, $\bar{m} = \hat{w}_2 \hat{m} \in \RR\axs_1 \rC \setminus \rC$.

Conversely, it is a contradiction to have $w \bar{m} \in \rC$ for some $w \in
\axs$ and $\bar{m} \in
\RR\axs_1 \rC \setminus \rC$ since $\bar{m} \not\in \rC$ would be a right chip
of $w \bar{m} \in \rC$.

To show uniqueness, suppose $w_1 \bar{m}_1 = w_2 \bar{m}_2 \in \RR\axs
\rC
\setminus \rC$, where $w_1,w_2 \in \axs$ and $\bar{m}_1, \bar{m}_2 \in \RR\axs_1
\rC \setminus \rC$, and suppose $|w_1| \leq |w_2|$.
If $|w_1| < |w_2|$, then decompose $w_2$ as $w_2 = w_1 u$ for some $u
\in \axs$ with $|u| > 0$.  Then $\bar{m}_1 = u \bar{m}_2$.  Decompose
$\bar{m}_1$ as $u_1 u_2\bar{m}_2$, where $u = u_1u_2$, $u_1, u_2 \in \axs$, and
$|u_1| = 1$, so that $u_2 \bar{m}_2 \in \rC$.
Therefore $\bar{m}_2 \in \rC$, which is a contradiction.
Therefore $|w_1| = |w_2|$, and by Lemma \ref{lem:canFactorMon}, $w_1 =
w_2$ and $\bar{m}_1 =
\bar{m}_2$.
\end{proof}

\begin{lemma}
\label{lem:repNEW2}
 Let $\rC \subset \RR^{1 \times \ell}\axs$ be a right chip space.
A monomial $m \in \cM^{\ell \times \ell}$ is in the space $\rC^* \RR\axs \rC
\setminus \rC^*\RR\axs_1
\rC$ if and only if it can be expressed as
$m = \bar{m}_1^* w \bar{m}_2$, where $\bar{m}_1, \bar{m}_2 \in
\RR\axs_1 \rC \setminus \rC$ and $w \in \axs$. Further, this representation is
unique.
\end{lemma}

\begin{proof}
First, $m \in \rC^*\RR\axs \rC$ if and only if it can be expressed as $a^*bc$,
where $a, c \in \rC$ and $b \in \axs$.
Next, either $bc \in \RR\axs_1 \rC$, in which case $a^*bc \in \rC^*\RR\axs_1
\rC$,
or by Lemma \ref{lem:repRRaxsrC}, $bc$ can be uniquely expressed as $bc =
b_1(b_2c)$, where $b = b_1b_2$,
$b_2c \in \RR\axs_1 \rC \setminus \rC$, and $|b_1| > 0$.  Next, either $b_1^*a
\in \rC$,
in which case $a^*bc \in \rC^*\RR\axs_1 \rC$, or by Lemma \ref{lem:repRRaxsrC},
$b_1^*a$ can be uniquely expressed as $b_4^*(b_3^*a)$, where $b_1 = b_3b_4$ and
$b_3^*a \in \RR\axs_1 \rC \setminus \rC$.
In the case that $a^*bc \in \rC^* \RR\axs \rC \setminus \rC^*\RR\axs_1
\rC$, we
see that $a^*bc = (b_3^*a)^*b_4(b_2c)$, with $b_3^*a, b_2c \in \RR\axs_1 \rC
\setminus \rC$ and $b_4 \in \axs$.

Next, suppose $\bar{m}_1^*w \bar{m}_2 = t^*uv$, where $\bar{m}_1, \bar{m}_2 \in
\RR\axs_1 \rC \setminus \rC$, $t,v \in \RR\axs_1 \rC$, and $w,u \in \axs$.
If $|t| > |\bar{m}_1|$, then Lemma \ref{lem:canFactorMon} implies that
$\bar{m}_1$ right divides $t$ properly, which implies by Lemma
\ref{lem:canFactorMon} that $t \not\in \RR\axs_1 \rC$, which is a contradiction.
 Therefore, $|t| \leq |\bar{m}_1|$, and similarly, $|v| \leq |\bar{m}_2|$.
If $|t| < |\bar{m}_1|$, then by Lemma \ref{lem:canFactorMon}, $t$ properly
divides $\bar{m}_1$ on the right, so Lemma \ref{lem:repRRaxsrC} implies that $t
\in \rC$.
Similarly, if $|v| < |\bar{m}_2|$, then $v \in \rC$.
In the case that $t,v \in \RR\axs_1 \rC \setminus \rC$, we see that $|t| =
|\bar{m}_1|$ and $|v| = |\bar{m}_2|$, which implies, by Lemma
\ref{lem:canFactorMon}, that $t = \bar{m}_1$ and $v = \bar{m}_2$, and further,
$w = u$.  This gives uniqueness.
Also, notice that
$\bar{m}_1^*w \bar{m}_2 \not\in \rC^*\RR\axs_1 \rC$ since if $t, v \in \rC$,
then $|t| < |\bar{m}_1|$, $|v| < |\bar{m}_2|$, and so $|u| \geq 2$.
\end{proof}

\section{\texorpdfstring{$\rC$}{[C]}-Orders and $\rC$-Bases}
\label{sec:COrders}

Now we turn to orders on monomials.
This is a subject familiar to those who work with Gr\"obner bases.
However, it turns out that to take advantage of the structure of
right chip spaces, we need to define an order different from the
admissible orders used with Gr\"obner bases (see $\S$
\ref{sect:LGB} for more on left admissible orders.)

Given a total order $\prec$ on $\cM^{\nu \times \ell}$, we say that the
\textbf{leading monomial}
\index{leading monomial} of a polynomial $p$, denoted $\lead{p}$, is the highest
monomial
appearing in $p$ according to $\prec$.  We call a polynomial
\textbf{monic}\index{monic polynomial} if its leading monomial has coefficient
$1$.  Given a set $I \subset \RR^{1 \times \ell}$, let $\lead{I}$ be the set of
leading monomials of elements of $I$, and let $\nonlead{I} := \cM^{1 \times
\ell} \setminus \lead{I}$.

\begin{lemma}
\label{lem:leadOfSumGen}
 Let $\prec$ be a total order on monomials in
$\RR^{\nu \times \ell}\axs$, and let $p_1, \ldots, p_k \in \RR^{\nu \times
\ell}\axs \setminus \{0\}$.  Suppose $\lead{p_1} \prec \cdots \prec \lead{p_k}$.
Then the leading monomial of $p_1 + \cdots + p_k$ is $\lead{p_k}$.
In particular, $p_1 + \cdots + p_k \neq 0$.
\end{lemma}

\begin{proof}
 Straightforward.
\end{proof}

For right chip spaces $\rC$, we would like to find an order on $\cM^{1 \times
\ell}$ such $a \prec b$ whenever $a \in \rC$ and $b
\not\in \rC$. To do this, we introduce $\rC$-orders.

First, we say a total order $\prec$ on $\axs$ is a \textbf{degree
order}\index{order!degree}
if $a \prec b$ holds whenever $|a| < |b|$.

Next, let $\rC \subset \RR^{1 \times \ell}\axs$ be a right chip space.
Let $\prec_0$ be a degree order on $\axs$.
We say that $\prec_{\rC}$ is a \textbf{$\rC$-order
(induced by $\prec_0$)}\index{C-order@$\rC$-order}
if $\prec_{\rC}$ is a total order on $\cM^{1 \times \ell}$ such that if $a,b \in
\cM^{1 \times \ell}$, then $a \prec_{\rC} b$ if
one of the following hold
\begin{enumerate}
 \item $a \in \rC$ and $b \not\in \rC$,
 \item $a \in \RR\axs \rC$ and $b \not\in \RR\axs \rC$,
 \item $a = a_1a_2$, $b = b_1b_2$, where $a_2,b_2 \in \RR\axs_1 \rC \setminus
\rC$, $a_1, b_1 \in \axs$, and $a_1 \prec_0 b_1$,
\item $a = wa_2$, $b = wb_2$, where $a_2,b_2 \in \RR\axs_1 \rC \setminus
\rC$, $w \in \axs$, and $a_2 \prec_{\rC} b_2$.
\end{enumerate}
The above conditions in and of themselves only define a partial
order. By definition, a $\rC$ order $\prec_{\rC}$ is defined in some way
among
the elements
of $\rC$, $\RR\axs_1 \rC \setminus \rC$, and $\RR^{1 \times \ell} \setminus
\RR\axs \rC$ respectively to make it a total order.

\begin{lemma}
 \label{lem:leadOfProd}
Let $\rC \subset \RR^{1 \times \ell}\axs$ be a right chip space
and let $\prec_{\rC}$ be a $\rC$-order induced by a degree order $\prec_0$.
If $q \in \RR\axs_1 \rC \setminus \rC$ and $p \in \RR\axs \setminus
\{0\}$, then the leading monomial of $pq$ is $\lead{p}\lead{q}$, where
$\lead{p}$ is the leading
monomial of
$p$ according to $\prec_0$ and $\lead{q} \in \RR\axs_1 \rC \setminus \rC$ is the
leading monomial of $q$ according to $\prec_{\rC}$.
\end{lemma}

\begin{proof}
 First, $\lead{q} \in \RR\axs_1
\rC
\setminus \rC$ since $q \in \RR\axs_1 \rC
\setminus \rC$ and the elements of $\rC$ are less than the elements of
$\RR\axs_1 \rC
\setminus \rC$ under $\prec_{\rC}$.
Consider the case where $p$ is a monomial.
Let $\phi \in \RR\axs_1 \rC$ be a monomial appearing in $q$.
If $\phi \in \rC$, then either $p\phi \in \rC$, or by Lemma
\ref{lem:repRRaxsrC}, $p\phi = p_1p_2 \phi$, where $p_1,p_2 \in \axs$ are such
that $p = p_1p_2$, and $p_2\phi \in \RR\axs_1 \rC \setminus \rC$.  In this
case, $|p_1| < |p|$.  In either case where $\phi \in \rC$, we see
that $p\phi \prec_{\rC} p\lead{q}$.
If, $\phi \in \RR\axs_1 \rC \setminus \rC$, but $\phi \neq
\lead{q}$, then since $\lead{q}$ is the leading monomial, $\phi \prec_{\rC}
\lead{q}$. Therefore
$p\phi \prec_{\rC} p\lead{q}$.  Therefore the leading monomial of $pq$ is
$p\lead{q}$.

The general case follows from Lemma \ref{lem:leadOfSumGen}.
\end{proof}

\subsection{\texorpdfstring{$\rC$}{[C]}-Bases}

Given a $\rC$-order and a left module $I$ generated by elements of $\RR\axs_1
\rC$, we can construct what we call a
$\rC$-basis.

 Let $\rC \subset \RR^{1 \times \ell}\axs$ be a right chip space
and let $\prec_{\rC}$ be a $\rC$-order.  Let $I \subset \RR^{1 \times
\ell}\axs$ be a left module generated by polynomials in $\RR\axs_1 \rC$.
We say that a pair of sets $(\{ \iota_i\}_{i \in A},
\{\vartheta_j\}_{j \in B})$ is a
\textbf{$\rC$-basis}\index{C-basis@$\rC$-basis} for $I$ if $\{ \iota_i\}_{i \in
A}$ is a maximal set of monic polynomials in
$I \cap
(\RR\axs_1 \rC \setminus \rC)$ with distinct leading monomials
and if $\{\vartheta_j\}_{j \in B}$
is a maximal (possibly empty) set of monic polynomials in $I \cap \rC$ with
distinct
leading monomials.

\begin{lemma}
\label{lem:strLinInd}
 Let $\rC \subset \RR^{1 \times \ell}\axs$ be a right chip space.
Let $\{ \iota_i\}_{i \in A} \subset \RR\axs_1 \rC \setminus \rC$
be a set of polynomials with distinct leading monomials,
and let $\{\vartheta_j\}_{j \in B} \subset \rC$ be a set of polynomials with
distinct
leading monomials.
The following are true of the polynomial $q$ defined by
\[
q = \sum_{i}^{\finite} p_i \iota_i + \sum_{j}^{\finite} \alpha_j
\vartheta_j,
\]
where each $p_i \in \RR\axs$ and $\alpha_j \in \RR$.
\begin{enumerate}
\item If $q \in \RR\axs \rC \setminus \rC$, then $\lead{q} =
\lead{p_i}\lead{\iota_i}$ for
some $i$.
\item If $q \in \RR\axs_1 \rC$, then each $p_i$ is
constant.
\item If $q \in \rC$, then each $p_i = 0$.
 \item \label{it:linIndReg} If $q = 0$, then each $p_i = 0$ and each
$\alpha_j = 0$.
\end{enumerate}
\end{lemma}

\begin{proof}
Fix a $\rC$-order.
If any of the $p_i$ are nonzero, by Lemmas \ref{lem:leadOfSumGen} and
\ref{lem:leadOfProd}, since all of the $\lead{q_i}$ are distinct,
the leading monomial of $q$ is the maximal
$\lead{p_i}\lead{q_i}$.  The maximal $\lead{p_i}$ has maximal length
by definition of a $\rC$-order.
If $q \in \RR\axs_1 \rC$, then so is its leading
monomial, in which case any nonzero $\lead{p_i}$ must be equal to $1$ by
Lemma \ref{lem:repRRaxsrC}. Therefore,
if  $q \in \RR\axs_1 \rC$, then each $p_i$ is
either $0$ or is a nonzero constant.
Further, if $q \in \rC$, then its leading monomial cannot
be in $\RR\axs_1 \rC \setminus \rC$.  In this
case, all of the $p_i = 0$.
Finally, if $q = 0$, then each of the $p_i = 0$ and further,
by Lemma \ref{lem:leadOfSumGen}, the $\alpha_j$ must all be $0$ as well.

\end{proof}

Note that Lemma \ref{lem:strLinInd} (\ref{it:linIndReg}) implies that the union
of the two sets
of a $\rC$-basis is a linearly independent set.

\begin{lemma}
\label{lem:niceCBasis}
Let $\rC \subset \RR^{1 \times \ell}\axs$ be a finite right chip space
and let $\prec_{\rC}$
be a $\rC$-order induced by some degree order.
 Let $I \subset \RR^{1 \times \ell}\axs$ be a left module
generated by polynomials in $\RR\axs_1 \rC$ and
let $(\{\iota_i\}_{i=1}^{\mu},
\{\vartheta_j\}_{j=1}^{\sigma})$ be a $\rC$-basis for $I$.
Each element of $I$ can be represented uniquely as
\begin{equation}
 \label{eq:formOfAllInI}
\sum_{i=1}^{\mu} p_i \iota_i + \sum_{j=1}^{\sigma} \alpha_j
\vartheta_j,
\end{equation}
where each $p_i \in \RR\axs$ and $\alpha_j \in \RR$.

Conversely, any pair of sets of monic polynomials $(\{\iota_i\}_{i=1}^{\mu},
\{\vartheta_j\}_{j=1}^{\sigma})$ with distinct leading monomials such that any
element of $I$
can be expressed
in the form (\ref{eq:formOfAllInI}) is a $\rC$-basis for $I$.
\end{lemma}

\begin{proof} Every element of $I \cap \RR\axs_1 \rC$ has leading monomial
equal to the leading monomial of an element of the $\rC$-basis, hence it
follows that the $\rC$-basis spans $I \cap \RR\axs_1 \rC$.  Further, if
$\vartheta_j \in I \cap \rC$, then for each $w \in \axs$ of length $1$ we
have $w
\vartheta_j \in \RR\axs_1 \rC$, which is the span of the $\rC$ basis. Since the
$\rC$-basis generates $I$,
this
implies that every element of $I$ can be expressed in the form
(\ref{eq:formOfAllInI}).  Further, uniqueness follows from Lemma
\ref{lem:strLinInd}.

Conversely, suppose
$(\{\iota_i\}_{i=1}^{\mu},
\{\vartheta_j\}_{j=1}^{\sigma})$ is a pair
of sets of monic polynomials with distinct leading monomials such
that any
element of $I$
can be expressed
in the form (\ref{eq:formOfAllInI}).
Let $\theta \in I \cap \RR\axs_1 \rC$
be equal to
\[\theta = \sum_{i=1}^{\mu} p_i \iota_i + \sum_{j=1}^{\sigma} \alpha_j
\vartheta_j.\]
Lemma \ref{lem:strLinInd} implies that if $\theta \neq 0$
then it cannot have a distinct leading monomial
from the $\iota_i$ and $\vartheta_j$.
Therefore the pair $(\{\iota_i\}_{i=1}^{\mu},
\{\vartheta_j\}_{j=1}^{\sigma})$ is a $\rC$-basis.
\end{proof}

\section{Double \texorpdfstring{$\rC$}{[C]}-Orders}
\label{sec:DoubleCOrders}

In addition to ordering elements of $\cM^{1 \times \ell}$,
we also want to order elements of $\cM^{\ell \times \ell}$.

Let $\rC \subset \RR^{1 \times \ell}\axs$ be a right chip space.
Let $\prec_{\rC}$ be a $\rC$-order.
We say that $\prec_{\rC \times \rC}$ is a
\textbf{double $\rC$-order (induced by
$\prec_{\rC}$)}\index{C-order@$\rC$-order!double}
if $\prec_{\rC \times \rC}$ is a total order on $\cM^{\ell \times \ell}$
such that given $a,b \in \cM^{\ell
\times \ell}$ we have $a \prec_{\rC
\times \rC} b$ if
one of the following hold
\begin{enumerate}
 \item $a \in (\RR^{1 \times \ell} \cap \rC)^* \rC$ and $b \not\in (\RR^{1
\times \ell} \cap \rC)^* \rC$,
 \item $a \in \rC^*\RR\axs \rC$ and $b \not\in \rC^*\RR\axs \rC$,
 \item $a = a_1^*a_2$, $b = b_1^*b_2$, where $a_2, b_2 \in \RR\axs_1 \rC
\setminus \rC$, $a_1, b_1 \in \RR\axs \rC$, and either
\begin{enumerate}
 \item $a_1 \prec_{\rC} b_1$, or
 \item $a_1 = b_1$ and $a_2 \prec_{\rC} b_2$.
\end{enumerate}
\end{enumerate}
The above conditions in and of themselves only define a partial
order. By definition, $\prec_{\rC \times \rC}$ is a total order, so it is
defined in some way
beyond
what has been stated to produce a total order.

\begin{lemma}
 \label{lem:leadCOrder2}
Let $\rC \subset \RR^{1 \times \ell}\axs$ be a right chip space, let
$\prec_{\rC}$ be a $\rC$-order, and let $\prec_{\rC \times \rC}$
be a double $\rC$-order induced by $\prec_{\rC}$.
If $q \in \RR\axs_1 \rC \setminus \rC$ and $p \in \RR\axs \rC \setminus
\rC$, then the leading monomial of $p^*q$ is $\lead{p}^*\lead{q}$, where
$\lead{p}$ is the
leading
monomial of
$p$ according to $\prec_{\rC}$ and $\lead{q} \in \RR\axs_1 \rC \setminus \rC$ is
the
leading monomial of $q$ according to $\prec_{\rC}$.
\end{lemma}

\begin{proof}
 First note that $\lead{q} \in \RR\axs_1 \rC \setminus \rC$
and $\lead{p} \in \RR\axs \rC \setminus \rC$
by definition of $\rC$.  By Lemma \ref{lem:repNEW2},
$\lead{p}^*\lead{q}$ is in $\rC^*\RR\axs \rC \setminus \rC^* \RR\axs_1 \rC$.
Let $\phi$ be a term of $p$ and let $\psi$ be a term of $q$.
Consider two cases.

If $\psi \in \rC$, let $\phi = e_i \otimes w$.  Then either $w \psi \in \rC$, 
in which case $\phi^*\psi \in (\RR^{1 \times \ell} \cap \rC)^* \rC$,
or by Lemma \ref{lem:repRRaxsrC}, $w^* \psi = w_2^* w_1^* \psi$, where $w_1^* 
\psi \in \RR\axs_1 \rC \setminus \rC$.
In the first case, by definition $\phi^* \psi \prec_{\rC \times \rC} 
\lead{p}^*\lead{q}$. In the second case, 
either $e_i \otimes  w_2 \in \rC$ so that $e_i \otimes  w_2 \prec_{\rC} 
\lead{p}$ or $e_i \otimes  w_2 \in \RR\axs \rC \setminus \rC$, implying that 
$e_i \otimes  w_2 \prec_{\rC} e_i \otimes _{\rC} w_1w_2 \preceq_{\rC} \lead{p}$.
Either way, $\phi^* \psi = (e_i \otimes w_2)^* (w_1^* \psi) \prec_{\rC \times 
\rC} 
\lead{p}^*\lead{q}$.

The second case is that $\psi \in \RR\axs \rC \setminus \rC$. In this case, it 
must be that $\psi \in \RR\axs_1 \rC \setminus \rC$ since $q \in \RR\axs_1 
\rC \setminus \rC$. It follows easily that $\phi^*\psi 
\preceq_{\rC \times \rC} \lead{p}^*\lead{q}$ and is only equal if $\phi = 
\lead{p}$ and $\psi = \lead{q}$.
\end{proof}

\begin{lemma}
\label{lem:niceCBasis2}
Let $\rC \subset \RR^{1 \times \ell}\axs$ be a finite right chip space,
let
$\prec_{\rC}$
be a $\rC$-order, and let $\prec_{\rC \times \rC}$ be a double $\rC$-order
induced by $\prec_{\rC}$. Let $I \subset \RR^{1 \times \ell}\axs$ be a left
module
generated by polynomials in $\RR\axs_1 \rC$.
Let $(\{\iota_i\}_{i=1}^{\mu},
\{\vartheta_j\}_{j=1}^{\sigma})$ be a $\rC$-basis for $I \cap
\RR\axs_1 \rC$.  Let $\{\tau_1, \ldots, \tau_{\omega}\} = \nonlead{I} \cap
\RR\axs_1 \rC \setminus \rC$.
\begin{enumerate}
 \item Every element of
$(\RR^{\ell \times 1} I + I^* \RR^{1 \times \ell}) \cap \rC^*\RR\axs \rC$ can be
represented as
\begin{align}
 \label{eq:repOfIIstar}
\sum_{i=1}^{\mu}\sum_{j=1}^{\mu} \iota_i^* p_{ij}^* \iota_j
+ \sum_{i=1}^{\omega}\sum_{j=1}^{\mu} \tau_i^* q_{ij}^* \iota_j
+ \sum_{i=1}^{\mu}\sum_{j=1}^{\omega} \iota_i^* r_{ij}^*  \tau_j\\
\notag
+ \sum_{i=1}^{\mu} (s_i^*\iota_i + \iota_i^*t_i)
+ \sum_{i=1}^{\sigma} (\alpha_i^*\vartheta_i + \vartheta_i^*\beta_i),
\end{align}
where $p_{ij}, q_{ij}, r_{ij} \in \RR\axs$, $s_i, t_i \in \rC$,
and
$\alpha_{j}, \beta_{j} \in \RR^{1 \times \ell} \cap \rC$.
Further, the $p_{ij}, q_{ij}$ and $r_{ij}$ are unique.
\item If (\ref{eq:repOfIIstar}) is in $\rC^*\RR\axs
\rC \setminus \rC^* \RR\axs_1 \rC$,  then its the leading polynomial is of the
form $m^*\lead{\iota_i}$ or
$\lead{\iota_i}^*m$
for some $\iota_i$ and some monomial $ m \in \RR\axs \rC \setminus \rC$.
\item \label{it:canRedC}If (\ref{eq:repOfIIstar}) is in $\rC^* \RR\axs_1 \rC$,
then
$p_{ab} = q_{cd} = r_{ef} = 0$ for each $1 \leq a,b,d,e \leq \mu$ and $1 \leq
c,f \leq \sigma$.
\end{enumerate}
\end{lemma}

\begin{proof}
 By Lemma \ref{lem:noTheta},
$(\RR^{\ell \times 1} I + I^* \RR^{1 \times \ell}) \cap \rC^*\RR\axs \rC =
\rC^*I + I^*\rC$, and consists of all polynomials of the form
\begin{equation}
 \label{eq:star}
\sum_{i=1}^{\mu} (a_{i}^* \iota_i +\iota_i^*b_i)
+ \sum_{j=1}^{\sigma} (\alpha_j^*\vartheta_j + \vartheta_j^*\beta_j),
\end{equation}
where $a_i, b_i, \alpha_j, \beta_j \in \RR\axs \rC$.
If $i \in \Gamma(\rC)$ and $w \in \axs$, either $|w| = 0$
or $w = w_1w_2$, where $w_1, w_2 \in \axs$ with $|w_1| = 1$.
In the second case, for each $\vartheta_j$,
we see that $(e_i \otimes
w_1w_2)^*\vartheta_j = (e_i \otimes w_2)^*(w_1^*\vartheta_j)$.
Since $w_1^*\vartheta_j \in I \cap \RR\axs_1 \rC$, then
$w_1^*\vartheta_j$ is in the span of the $\rC$-basis.
Reducing in this way, we can take the $\alpha_i$ and $\beta_i$ in
(\ref{eq:repOfIIstar}) to be elements of $(\RR^{1 \times \ell} \otimes 1) \cap
\rC$.

Next, consider
\[\sum_{i=1}^{\mu} (a_{i}^* \iota_i +\iota_i^*b_i).\]
If $\lead{a_i} \in \RR\axs \rC
\setminus \rC$, then by Lemma \ref{lem:repRRaxsrC}, it can be decomposed as
$\lead{a_i} = w m$, where $w \in \axs$ and $m \in \RR\axs_1 \setminus \rC$.
Let $A \in \RR$ be the coefficient of $\lead{a_i}$ in $a$.
By construction, $m$ is either equal to some $\lead{\iota_j}$ or some $\tau_j$.
In the first case, we see that $a_i = (a_i - A w\iota_j) + w\iota_j$, so that
$a_i - A w\iota_j$ has a smaller leading monomial than $a_i$.
In the latter case, $a_i = (a_i - w\tau_j) + w\tau_j$, so that
$a_i - w\tau_j$ has a smaller leading monomial than $a_i$.
Since $\rC$ is finite dimensional, we can continue this until we decompose each
$a_i$,
and likewise each $b_i$, as
\[ a_i = \sum_{j=1}^{\mu} \tilde{a}_{ij} \iota_j + \sum_{j=1}^{\sigma}
q_{ij}\tau_j + s_i \quad
\mbox{and}
\quad
b_i = \sum_{j=1}^{\mu} \tilde{b}_{ij} \iota_j + \sum_{j=1}^{\sigma}
r_{ij}^* \tau_j + t_i,
\]
where $\tilde{a}_{ij}, q_{ij}, \tilde{b}_{ij}, r_{ij} \in \RR\axs$ and $s_i,
t_i \in \rC$.
Therefore
\begin{align}
\notag
\sum_{i=1}^{\mu} (a_{i}^* \iota_i +\iota_i^*b_i)
&=
\sum_{i=1}^{\mu}\sum_{j=1}^{\mu} \iota_i^* (\tilde{a}_{ij}^* + \tilde{b}_{ij})
\iota_j
+ \sum_{i=1}^{\omega}\sum_{j=1}^{\mu} \tau_i^* q_{ij}^* \iota_j\\
\notag
&+ \sum_{i=1}^{\mu}\sum_{j=1}^{\omega} \iota_i^* r_{ij}^*  \tau_j
+ \sum_{i=1}^{\mu} (s_i^*\iota_i + \iota_i^*t_i).
\end{align}
Setting $p_{ij} = \tilde{a}_{ij} + \tilde{b}_{ij}^*$ gives
(\ref{eq:repOfIIstar}).

Next, consider the leading monomial of (\ref{eq:repOfIIstar}).
If $p_{ij} \neq 0$ for some $ij$, then by Lemmas \ref{lem:leadOfProd} and
\ref{lem:leadCOrder2},
the leading monomial of $\iota_i^*p_{ij}^*\iota_j$ is
\[\lead{\iota_i^*p_{ij}^*\iota_j} = \lead{p_{ij}\iota_i}^*\lead{\iota_j}
= \lead{\iota_i}^*\lead{p_{ij}}^*\lead{\iota_j}.\]
Similarly, the leading monomials of each nonzero $\tau_i^*q_{ij}^*\iota_j$ and
$\iota_i^*r_{ij}^*\tau_j$ are
\[\lead{\tau_i^*q_{ij}^*\iota_j}
= \tau_i^*\lead{q_{ij}}^*\lead{\iota_j}
\]
and
\[
\lead{\iota_i^*r_{ij}^*\tau_j}
= \lead{\iota_i}^*\lead{r_{ij}}^*\tau_j.
\]
All of these possible leading monomials are distinct by Lemma \ref{lem:repNEW2}.
The last line in (\ref{eq:repOfIIstar}) is in $\rC^* \RR\axs_1 \rC$, so if 
(\ref{eq:repOfIIstar}) is not in $\rC^* \RR\axs_1 \rC$, then its leading 
monomial must be of the form $m^*\lead{\iota_i}$ or $\lead{\iota_i}^*m$ for 
some $\iota_i$.  On the other hand, if 
the last line in (\ref{eq:repOfIIstar}) is in $\rC^* \RR\axs_1 \rC$, then we see 
that each $p_{ab} = q_{cd} = r_{ef} = 0$ since otherwise its leading monomial 
would lie outside of $\rC^*\RR\axs_1 \rC$.
By linearity, this implies that the representation in (\ref{eq:repOfIIstar}) 
must be unique.
\end{proof}

\section{Positive Linear Functionals on \texorpdfstring{$\RR^{\ell \times
\ell}\axs$}{[R ell times ell x xs]}}
\label{sect:linFun}

The main theorem of this paper, Theorem \ref{thm:lnss}, will be proved at the
end of this section.  We now discuss linear functionals and prove some
important lemmas which will be used in the proof of the main result.

A ($\RR$-)linear functional $L$ on $W \subset \RR^{\ell \times \ell}\axs$ is
\textbf{symmetric}\index{linear functional on $\RR^{\ell \times
\ell}$!symmetric}
if $L(\omega^{\ast}) = L(\omega)$ for each pair $\omega, \omega^*
\in W$.
A linear functional $L$ on a subspace $W \subset \RR^{\ell \times
\ell}\axs$ is
\textbf{positive}\index{linear functional on $\RR^{\ell \times
\ell}\axs$!positive} if it is symmetric and
if $L(w^{\ast}w) \geq 0$ for each  $w^*w \in W$.

\begin{lemma}
\label{prop:posLinFunHasIdeal}
 If $L$ is a positive linear functional on $\RR^{\ell \times
\ell}\axs$, then
$I := \{\vartheta \in \RR^{1 \times \ell}\axs \mid
L(\vartheta^{\ast}\vartheta) = 0 \}$
is a real left module and $L(\RR^{\ell \times 1} I + I^{\ast}\RR^{1 \times
\ell}) =
\{0\}$.
\end{lemma}

\begin{proof}
The space $\RR^{\ell \times 1} I$ is spanned by polynomials of the form
$b^{\ast}a$,
where $a \in I$ and $b \in \RR^{1 \times \ell}$.  For each $\xi \in
\RR$, by positivity of $L$,
\[ L([a + \xi b]^{\ast}[a + \xi b]) = 2 \xi L(b^{\ast}a)
+
\xi^2 L(b^{\ast}b)\geq 0,\]
which implies that $L(b^{\ast}a) = 0$. Therefore,
$L(\RR^{\ell
\times 1} I
+
I^{\ast}\RR^{1 \times \ell}) =
\{0\}$.
Next, if $q \in \RR\axs$, then for
each $\xi \in \RR$,
\[ L([a + \xi q^{\ast}qa]^{\ast}[a + \xi q^{\ast}qa]) = 2 \xi
L(a^{\ast}q^{\ast}qa) + \xi^2 L(a^{\ast}q^{\ast}qq^{\ast}qa)
\geq 0,\]
which implies that $L(a^{\ast}q^{\ast}qa) = 0$.
Therefore $qa \in I$.
Further, if $c \in I$, then for each $\xi \in \RR$,
\[ L([a + \xi c]^{\ast}[a + \xi c]) = 2\xi L(a^{\ast}c)
 \geq 0, \]
which implies that $L(a^{\ast}c)= 0$. Therefore, $a + c
\in I$, which implies that $I$ is a left module.
Finally suppose $\sum_{i}^{\finite} p_i^{\ast}p_i \in \RR^{\ell \times 1} I +
I^{\ast}\RR^{1 \times \ell}$.
Then
\[L\left( \sum_{i}^{\finite} p_i^{\ast}p_i\right) = \sum_i^{\finite}
L(p_i^{\ast}p_i)= 0,\]
which implies that each $L(p_i^{\ast}p_i) = 0$.  Therefore
each
$p_i
\in I$, which implies that $I$ is real.
\end{proof}

Let $W \subset \RR^{\ell \times \ell}\axs$ be a vector subspace and
let $L$ be a positive linear functional on $W$.
Suppose
\[ \{\omega \in \RR^{1 \times \ell}\axs \mid \omega^*\omega \in W \} = J \oplus
T \]
where $J, T \subset \RR^{1 \times \ell}\axs$ are
vector
subspaces with
\[J := \{\vartheta \in \RR^{1 \times \ell}\axs \mid \vartheta^*\vartheta \in W
\mbox{ and } L(\vartheta^*\vartheta) = 0\}.\]
 An extension $\bar{L}$ of
$L$ to a space $U \supset W$ is a \textbf{flat
extension}\index{flat extension} if $\bar{L}$ is positive and if
\[ \{u \in \RR^{1 \times \ell}\axs \mid u^*u \in U \} = I \oplus
T \]
where
\[ I = \{ \iota \in \RR^{1 \times \ell}  \mid \iota^*\iota \in U \mbox{ and }
\bar{L}(\iota^*\iota) = 0\}.\]

\begin{prop}
\label{prop:flatExtRC}
Let $\rC \subset \RR^{1 \times \ell}\axs$
be a finite right chip space.
Let $L$ be a positive linear functional on
$\rC^*\RR\axs_1 \rC$.
\begin{enumerate}
 \item There exists a positive extension of
$L$ to the space  $\rC^*\RR\axs_2 \rC$ if and only if
whenever
 $\vartheta \in \rC$ satisfies
$L(\vartheta^*\vartheta) = 0$, then $L(b^*c\vartheta) =
0$ for each polynomial $b \in \RR\axs_1 \rC$ and each $c \in \RR\axs$
such that $c\vartheta \in \rC$.
\item If there exists a positive extension of
$L$ to the space
$\rC^*\RR\axs_2 \rC$, then there exists a unique flat
extension
$\bar{L}$  of $L$
to
$\rC^*\RR\axs \rC$.
In this case, the space
\[\{\theta \in \RR\axs\rC \mid
\bar{L}(\theta^*\theta) = 0\}\]
is generated as a left module by the set
\begin{equation}
\label{eq:genOfJRC}
  \{ \iota \in \RR\axs_1 \rC \mid L(b^*\iota) = 0
\text{ for every } b \in \rC\}.
\end{equation}
\item Given the existence of a flat extension $\bar{L}$ to
$\rC^*\RR\axs\rC$, there exists a flat extension
of $\bar{L}$ to all of $\RR^{\ell \times \ell}\axs$.
\end{enumerate}

\end{prop}

\begin{proof}
First, suppose there exists a positive extension $\tilde{L}$ of
$L$ to the space
$\rC^*\RR\axs_2\rC$.
Let $\vartheta \in
\rC$ satisfy $L(\vartheta^*\vartheta)
= 0$.
If $\tilde{c} \vartheta \in \rC$, with $\tilde{c} = \RR\axs_1$ then
$\tilde{c}^*\tilde{c}\vartheta \in \RR\axs_1
\rC$.  For each
$\xi \in \RR$,
since $\tilde{L}$ is positive,
\[ \tilde{L}([\vartheta + \xi \tilde{c}^*\tilde{c}\vartheta]^*[\vartheta + \xi
\tilde{c}^*\tilde{c}\vartheta])
= 2\xi L (\vartheta^*\tilde{c}^*\tilde{c}\vartheta) + \xi^2
\tilde{L}(\vartheta^*\tilde{c}^*\tilde{c}\tilde{c}^*\tilde{c}\vartheta) \geq
0,\]
which implies that $L([\tilde{c}\vartheta]^*[\tilde{c}\vartheta]) = 0$.
We can then extend this to show that if $c \vartheta \in \rC$ for any $c \in
\RR\axs$ then $L([c\vartheta]^*[c\vartheta]) = 0$.
 Further, if $b \in \RR\axs_1 \rC$, then for each $\xi \in \RR$,
since $\tilde{L}$ is positive,
\[ \tilde{L}([c\vartheta + \xi b]^*[c\vartheta + \xi b])
= 2\xi L (b^*c\vartheta) + \xi^2 \tilde{L}(b^*b) \geq 0,\]
which implies that $L(b^*c\vartheta) = 0$.

Conversely,
let $\rC$ be decomposed as $J \oplus T$, where
\[J = \{ \vartheta \in \rC \mid L(\vartheta^*\vartheta) = 0\},\]
and $T \subset \rC$ is some complementary subspace,
and suppose
$L(b^*c\vartheta) = 0$ for each $\vartheta \in J$, each $b \in \RR\axs_1
\rC$, and each $c \in \RR\axs$
such that $c\vartheta \in \rC$.

Define an inner product on $T$ by $\langle a, b \rangle = L(b^*a)$.
This inner product is well defined since $L$ is symmetric and is positive
on squares of $T$.
 Let $\tau_1, \ldots, \tau_{\mu}$ be an orthonormal basis for $T$ according to
this inner product.
If $m \in
\RR\axs_1 \rC \setminus \rC$ is a monomial, then
consider the polynomial $p_m$ defined by
\begin{equation}
 \label{eq:pwgiRC}
p_{m}:= m - \sum_{j=1}^{\mu} L(\tau_j^*m) \tau_j.
\end{equation}
If $b \in \rC$, then $b = \vartheta + \sum_{k=1}^{\mu} \beta_k \tau_k$ for
some
$\vartheta \in J$ and $\beta_k \in \RR$.  We see that
\begin{align}
\notag
 L( b^* p_{m}) &= L(\vartheta^*p_m)
 +
\sum_{k=1}^{\mu}\beta_k L(\tau_k^*m)
 - \sum_{j=1}^{\mu}
\sum_{k=1}^{\mu} \beta_k L(\tau_j^*m)L(\tau_k^*\tau_j) = 0.
\end{align}
Therefore $p_{m}$ is in the set (\ref{eq:genOfJRC}).
Also note that a
flat extension $\bar{L}$ of $L$
must satisfy $\bar{L}(p_{m}^*p_{m}) = 0$ since the equation
\[ \bar{L}\left(\left[p_{m} + \sum_{j=1}^{\mu} \gamma_j
\tau_j\right]^*
\left[p_{m} + \sum_{j=1}^{\mu} \gamma_j
\tau_j\right]\right) =
\sum_{j=1}^{\mu} \gamma_j^2 L(\tau_j^*\tau_j) +
\bar{L}(p_{m}^*p_{m}) = 0\]
has only one solution in $\gamma$ and $\bar{L}(p_m^*p_m)$: each
$\gamma_i = 0$ and $\bar{L}(p_{m}^*p_m) = 0$.

Next, consider $w\psi \in \RR\axs_d \rC$, where $w = w_1w_2$, with
$|w_2|
=
1$, and $\psi \in \rC$.  We see that either $w_2 \psi \in \rC$,
or $w_2 \psi \in \RR\axs_1 \rC \setminus \rC$.  In either case, $w_2 \psi =
\iota + \zeta$,
where $\iota \in \operatorname{Span}(\{p_m \mid m \in \RR\axs \rC
\setminus \rC \})$ and $\zeta \in \rC$.  Therefore
\[ w \psi = w_1 \iota + w_1 \zeta,\]
so that $w_1\iota \in \RR\axs \left(\operatorname{Span}[\{p_m \}]
\right)$ and $w_1\zeta \in \RR\axs_{d-1} \rC$.
By induction, this implies that each element of $\RR\axs\rC$ is in $\RR\axs
(\operatorname{Span}[\{p_m \}]) + \rC$.
 Lemma \ref{lem:strLinInd} further implies that
\[  \RR\axs\rC = \RR\axs (\operatorname{Span}[\{p_m \}]) \oplus \rC. \]

Let $I$ be the left module generated by $(\RR\axs \operatorname{Span}[\{p_m \}])
\oplus J$.
Let $\iota$ be in the set (\ref{eq:genOfJRC}).  Since $\iota \in \RR\axs_1 \rC$
we can decompose it as
\[\iota = \sum_{m \in \RR\axs \rC \setminus \rC} \alpha_m p_m + \vartheta,\]
where each $\alpha_m \in \RR$ and $\vartheta \in \rC$.
We see that
$L(\vartheta^*\vartheta) = L(\vartheta^*\iota) = 0$,
which implies that $\vartheta \in J$.  Therefore,
 $\iota \in I$, which implies that
$I$ contains the set (\ref{eq:genOfJRC}).

Conversely, let $\vartheta \in J$.  By assumption $\vartheta$ is in the set
(\ref{eq:genOfJRC}).
Further, if $c \in \axs$ has length $1$, then $c\vartheta \in \RR\axs_1 \rC$.
By assumption, we must have $c\vartheta$ in (\ref{eq:genOfJRC}).  Therefore
the generators of $I$ are in (\ref{eq:genOfJRC}), which implies that $I$ is the
left module
generated by (\ref{eq:genOfJRC}).
Further, since $c\vartheta$ is in the set
(\ref{eq:genOfJRC}),
$c\vartheta$ must be in the span of the $p_m$ and $J$.  Therefore
\[I = \RR\axs \operatorname{Span}(\{p_m \}) \oplus J
\quad
\mbox{ and }
\quad
\RR \axs \rC = I \oplus T.\]
Define an inner product linearly
on $\RR\axs
\rC / I$ to be
\[\langle [\tau_i], [\tau_j] \rangle := \langle
\tau_i, \tau_j \rangle = L(\tau_i^*\tau_j),\]
where $\tau_1, \ldots, \tau_{\mu}$ are the orthonormal basis elements of $T$
previously
defined.
Since $I \cap T = \{0\}$, this inner product is well defined.

Let $\bar{L}$ be a linear functional on $\rC^*\RR\axs
\rC$ defined by
\[\bar{L}(b^*a) := \langle [a], [b] \rangle.\]
Clearly $\bar{L}$ is positive and symmetric. Further, given an
element $b^*pa \in
\rC^*\RR\axs_1 \rC$, with $a,b \in \rC$ and $p \in
\RR\axs_1$,
decompose $pa$ as $pa = \iota_a + \tau_a$ and $b$ as $b = \vartheta_b + \tau_b$,
where
$\iota_a \in I$, $\vartheta_b \in J$, and $\tau_a,\tau_b \in T$.  Then
\[\bar{L}(b^*pa) = \langle [pa], [b]\rangle =
\langle [\tau_a], [\tau_b] \rangle = L(\tau_b^*\tau_a) =
L(b^*a),\]
since $L(\tau_b^*\iota_a) = L(\vartheta_b^*\tau_a) =
L(\iota_b^*\vartheta_a) =0$.  Therefore
$\bar{L}$ is an extension of $L$.
Further, $\bar{L}$ is a flat extension since $\RR\axs \rC =
I \oplus T$, and clearly
$I = \{ \theta \in \RR\axs \rC \mid \bar{L}(\theta^*\theta)
= 0\}$.  Finally, as mentioned, any flat extension of $L$ must
satisfy $\bar{L}(p_m^*p_m) = 0$ for each monomial $m \in \RR\axs_1 \rC
\setminus
\rC$,
so by Lemma
\ref{prop:posLinFunHasIdeal}, we must have $\bar{L}(\RR^{\ell \times
1}I + I^*\RR^{1 \times \ell}) = \{0\}$.
Therefore, $\bar{L}$ is unique since
$\RR\axs \rC = I \oplus T$, and the value of $\bar{L}$ on each of
$T^*T$, $(I \oplus T)^*I$ and $I^*T$ is uniquely determined.

Finally, we extend $\bar{L}$ to a flat extension
on all of $\RR^{\ell \times \ell}\axs$ as follows.
Lemma \ref{lemma:MsFFaxsM} implies that  $\rC^* \RR\axs \rC$ is equal to
\[ \rC^* \RR\axs \rC =  \bigoplus_{i,j \in \Gamma(\rC)} E_{i,j}
\otimes \RR\axs.\]
Extend $\bar{L}$
to be $0$ on the set
\[ \bigoplus_{(k_1,k_2) \not\in  \Gamma(\rC)^2 } E_{k_1k_2}
\otimes \RR\axs.\]
Clearly this is a flat extension of
$L$ to all of $\RR^{\ell \times \ell}\axs$.
\end{proof}

\subsection{The GNS Construction}

Proposition \ref{prop:GNS} below describes the well-known
 Gelfand-Naimark-Segal (GNS) construction.

\begin{prop}
\label{prop:GNS}
Let $L$ be a positive linear functional on
$\RR^{\ell \times \ell}\axs$,
and let $I = \{ \vartheta \in \RR^{1 \times \ell}\axs \mid
L(\vartheta^*\vartheta) =
0\}$.
There exists an inner product on the quotient space
$\cH := \RR^{1 \times \ell}\axs / I$,
a tuple of (possibly unbounded) operators $X$ on $\cH$, and a
vector $v \in \cH^n$ such that for each $p
\in \RR^{\ell \times \ell}\axs$ we have
\[\langle p(X)v, v \rangle = L(p).\]
 and $\cH = \{q(X)v \mid q \in\RR^{1 \times \ell}\axs\}$.
\end{prop}

\begin{proof}
Define an inner product on $\cH$ to
be
\[ \langle [p], [q] \rangle := L(q^{\ast}p).\]
This inner product is well defined since $L$ is positive,
and, by Lemma \ref{prop:posLinFunHasIdeal}, is $0$ on the space $\RR^{1
\times \ell} I + I^* \RR^{1 \times \ell}$.
Let $X$ be the tuple of operators on $\cH$ such that for each variable $x_k$,
the operator $X_k$ is defined by
$X_k [p]  :=
[x_k p]$.
Since $I$ is a left module by Lemma \ref{prop:posLinFunHasIdeal}, $X$ is well
defined.
Further,
\begin{align}
 \notag
 \left\langle X_k^{\ast}[p],
[q]
\right\rangle
&= \left\langle [p],
X_k [q]
\right\rangle =\left\langle [p],
[x_k q]
\right\rangle =L\left( q^*x_k^*p \right)=
\left\langle [x_k^*p],
[q]
\right\rangle.
\end{align}
Therefore
$X_k^{\ast}[p]
= [x_k^*p]$.
Further, it follows that for any $r \in \RR\axs$ that
$r(X)[p] =[rp]$.

Fix $v \in \cH^{\ell}$ to be
\[ v := \begin{pmatrix}
        [e_1 \otimes 1]\\
        \vdots \\
        [e_{\ell} \otimes 1]
       \end{pmatrix}.
\]
If $q = \sum_{i=1}^{\ell} e_i \otimes q_i \in \RR^{1 \times \ell}\axs$, then
\[q(X)v = \sum_{i=1}^{\ell} q_i(X)[e_i \otimes 1] = \sum_{i=1}^{\ell} [e_i
\otimes q_i]  = [q].\]
Therefore,
\[\cH = \{ q(X)v \mid q  \in \RR^{1 \times \ell}\axs\}.\]
If $p = \sum_{i=1}^{\ell}\sum_{j=1}^{\ell} E_{ij} \otimes p_{ij} \in \RR^{\ell
\times \ell}\axs$,  we see
\begin{align}
 \notag
p(X)v &=
\begin{pmatrix}
 \sum_{j=1}^{\ell} p_{1j}(X)[e_j \otimes 1]\\
\vdots\\
\sum_{j=1}^{\ell} p_{\ell j}(X)[e_j \otimes 1]
\end{pmatrix}
=\begin{pmatrix}
 \sum_{j=1}^{\ell} [e_j \otimes p_{1j}]\\
\vdots\\
 \sum_{j=1}^{\ell} [e_j \otimes p_{\ell j}]
\end{pmatrix}
\end{align}
so that
\begin{align}
 \notag
\langle p(X)v, v \rangle
&= \sum_{i=1}^{\ell} \left\langle
\sum_{j=1}^{\ell} [e_j \otimes p_{ij}], [e_i \otimes 1] \right\rangle\\
\notag
&= \sum_{i=1}^{\ell} \sum_{j=1}^{\ell} L(E_{ij} \otimes p_{ij})
= L(p).
\end{align}
\end{proof}

\begin{corollary}
\label{cor:GNS}
Let $\rC \subset \RR^{1 \times \ell}\axs$
be a finite right chip space.
Let $L$ be a positive linear functional on
$\rC^*\RR\axs_1 \rC$, and let
$J$ be the set
\[J := \{ \vartheta \in \rC \mid L(\vartheta^*\vartheta) = 0\}.\]
Further, suppose there exists a positive extension of $L$ to
$\rC^*\RR\axs_2 \rC$.

Let $n := \dim(\rC) -
\dim(J \cap \rC)$,
and suppose $n > 0$.
There exists a tuple $X$ of  $n \times n$ matrices over $\RR$ and a
vector $v \in \RR^{\ell n}$ such that for each $p
\in\rC^*\RR\axs_1\rC$ we have
\[v^{\ast}p(X)v = L(p).\]
 and $\RR^{\ell n} = \{p(X)v \mid p \in\rC\}$.
\end{corollary}

\begin{proof}
By Proposition \ref{prop:flatExtRC}, there exists a flat extension
$\bar{L}$ of
$L$ to all of $\RR^{\ell \times \ell}\axs$.
Given this flat extension, apply Proposition \ref{prop:GNS} to produce the
desired
$X$ and $v$.
\end{proof}

\subsection{Non-Commutative Hankel Matrices}

Let $\omega = (\omega_i)_{i=1}^k$ be a vector whose entries form a basis for a
vector space $W \subset \RR^{1 \times \ell}\axs$. Given a linear functional $L$
on $W^*W$, the {\bf non-commutative Hankel matrix for $L$ (with respect to
$\omega$)} is the matrix
$A = (L(\omega_i^*\omega_j))_{1 \leq i,j \leq k}$.
This concept is a non-commutative analog of moment matrices---see \cite{CF},
\cite{CF2}, for example.

Recall that $\mathbb{S}^{k}$ is the set of $k \times k$ symmetric matrices
over $\RR$. Define $\langle A, B \rangle := \operatorname{Tr}(AB)$ to be the
inner product
on $\mathbb{S}^{k}$.

\begin{lemma}
 \label{lem:Hankel}
 Let $\omega = (\omega_i)_{i=1}^k$ be a vector whose entries form a basis for a
vector space $W \subset \RR^{1 \times \ell}\axs$.  Let $A \in 
\mathbb{S}^{k}$
be a matrix.  Let $I \subset \RR^{1 \times \ell}\axs$ be a left module, and
define $\cZ$ to be
\[
 \cZ := \{ C \in \mathbb{S}^{k} \mid \omega^*C\omega \in \RR^{\ell
\times
1}I + I^*\RR^{1 \times \ell}\}
\]
Then $A$ is the non-commutative Hankel matrix for some symmetric linear 
functional $L$ on
$W^*W$ such that $L([\RR^{\ell \times 1}I + I^*\RR^{1 \times \ell}] \cap W^*W)
= \{0\}$ if and only if $A \in \cZ^{\bot}$, in which case
\begin{equation}
\label{eq:defLinFun}
L(\omega^* C \omega) = \operatorname{Tr}(AC)
\end{equation}
for each $C \in \RR^{k \times k}$.
\end{lemma}

\begin{proof}
 First, suppose $A = (a_{ij})_{1 \leq i,j \leq k}$ is the Hankel matrix of $L$
 and $L([\RR^{\ell \times 1}I + I^*\RR^{1 \times \ell}] \cap W^*W)
= \{0\}$.
Then given $C = (c_{ij})_{1 \leq i,j \leq k}$,
 \[
  L(\omega^*C\omega) = L\left( \sum_{i=1}^k \sum_{j=1}^k c_{ij}
\omega_i^*\omega_j \right) = \sum_{i=1}^k \sum_{j=1}^k a_{ij}c_{ij} =
\operatorname{Tr}(AC).
 \]
 It is therefore clear that $A \in \cZ^{\bot}$.

 Conversely, given $A = (a_{ij})_{1 \leq i,j \leq k} \in \cZ^{\bot}$,
(\ref{eq:defLinFun}) gives a well-defined linear functional since if $\omega^*C
\omega = 0 \in \RR^{\ell \times 1}I + I^*\RR^{1 \times \ell}$, then
$\operatorname{Tr}(AC) = 0$.  Further
\[
L(\omega_i^*\omega_j) = L(\omega^*E_{ij}\omega) = \operatorname{Tr}(AE_{ij}) =
a_{ij}.
\]
\end{proof}

\begin{prop}
\label{lem:posHankel}
Let $\rC \subset \RR^{1 \times \ell}\axs$ be a finite right chip space.
Let $I \subset \RR^{1 \times \ell}\axs$ be a left module generated by
polynomials in $\RR\axs_1 \rC$.
 Let $\tau = (\tau_i)_{i=1}^k$ be a vector whose entries are all elements of
$\nonlead{I} \cap \RR\axs_1\rC$, and let $T$ be the span of the $\tau_i$. Let
$L$ be a symmetric linear
functional on $T^*T$, and let $A \in \mathbb{S}^{k}$
be its non-commutative Hankel matrix.
Let $\cZ \subset \mathbb{S}^{k}$ be the space
\[
 \{ Z \in\mathbb{S}^{k} \mid \tau^* Z \tau \in \RR^{\ell \times 1} I + I^*
\RR^{1 \times \ell}\}
\]

Then $L$ can be extended to a positive linear functional $\overline{L}$ on
$\rC^*\RR\axs_2 \rC$ such that
\begin{enumerate}
 \item $\overline{L}(\iota) = L(\iota^*) = 0$ for each $\iota \in \RR^{\ell 
\times
1}I$
 \item $\overline{L}(a^*a)  = 0$ if and only if $a \in I$.
\end{enumerate}
if and only if $A \succ 0$ and $A \in \cZ^{\bot}$.
\end{prop}

\begin{proof}
First, if there exists such a $\overline{L}$, then it is clear from
(\ref{eq:defLinFun}) that $L(a^*a) > 0$ for each $a \in T$ if and only if $A
\succ 0$, and further, Lemma \ref{lem:Hankel} implies that
$A \in \cZ^{\bot}$.

Conversely, let $A \succ 0$ and $A \in \cZ^{\bot}$.
We see that $\rC^*\RR\axs_2 \rC = (\RR\axs_1 \rC)^*(\RR\axs_1 \rC)$, so it
suffices to define $\overline{L}$ on products $p_1^*p_2$, where $p_1, p_2 \in
\RR\axs_1 \rC$.  If $\iota \in I \cap \RR\axs_1 \rC$ and $p \in \RR\axs_1$,
define $\overline{L}(p^*\iota) = \overline{L}(\iota^*p) = 0$.  This agrees with 
the
definition of $L$ by Lemma \ref{lem:Hankel}, and Lemma \ref{lem:niceCBasis2}
implies that $\overline{L}(\theta) = \overline{L}(\theta^*) = 0$ for each 
$\theta \in
\RR^{\ell \times 1}I$ on which $\overline{L}$ is defined. Further, if $a^*a \in
\rC^*\RR\axs_2 \rC$, then Lemma \ref{lem:leadOfSquare} implies that each $a \in
\RR\axs_1 \rC$. Hence $a = \iota  + \beta^* \tau$, where $\iota \in
\RR^{\ell \times 1}_1I$ and $\beta \in \RR^k$ is a constant vector.  Since $A 
\succ 0$, by
(\ref{eq:defLinFun}) we see that $\overline{L}(a^*a) = \beta^* A \beta = 0$ if 
and
only if $\beta = 0$, which is equivalent to $a \in I$.
\end{proof}

\begin{lemma}
\label{lem:pSd}
 Let $\cB \subset \mathbb{S}^{k}$ be a vector subspace.
Then exactly one of the following holds:
\begin{enumerate}
 \item\label{item:1} There exists $B \in \cB$ such that $B \succ
0$, and there exists no nonzero $A \in \cB^{\bot}$ with $A \succeq 0$.
\item\label{item:2} There exists $A \in \cB^{\bot}$ such that $A \succ 0$, and
there exists no nonzero $B \in \cB$ with $B \succeq 0$.
\item There exist nonzero $B \in \cB$ and $A \in \cB^{\bot}$ with
$A, B \succeq
0$, but there exist no $B \in \cB$ nor $A \in \cB^{\bot}$ with either $A \succ
0$ or $B \succ 0$.
\end{enumerate}

\end{lemma}

\begin{proof}
  Let $B_1, \ldots, B_n$
be an orthonormal basis for $\cB$, and let $A_1, \ldots, A_m$ be an orthonormal
basis for $\cB^{\bot}$.
Define $L(\alpha, \beta)$, where $\alpha \in \RR^m$ and $\beta \in \RR^n$, to be
\[ L(\alpha, \beta) := \sum_{i=1}^m \alpha_i A_i + \sum_{j=1}^n \beta_j B_j.\]
The elements of $\cB$ are precisely all matrices of the form $L(0, \beta)$ and
the
elements of $\cB^{\bot}$ are precisely all matrices of the form $L(\alpha, 0)$.
For any pair $(\alpha, \beta)$,
\[ \langle L(\alpha, 0), L(0,\beta) \rangle = 0.\]
If $L(\alpha, 0) \succ 0$ for some $\alpha \in \RR^m$, then $L(0,\beta) 
\not\succeq 0$
for each $\beta \in \RR^n \setminus \{0\}$.
Similarly, if $L(0, \beta) \succ 0$ for some $\beta \in \RR^n$, then 
$L(\alpha, 0) \not\succeq 0$ for
each $\alpha \in \RR^m \setminus \{0\}$.
Therefore, either (\ref{item:1}) or (\ref{item:2}) holds, or there exist no $B
\in
\cB$ nor $A \in \cB^{\bot}$ with either $A \succ 0$ or $B \succ 0$.

Assume that (\ref{item:1}) and (\ref{item:2}) do not hold.  Let $\cC
\subset \RR^n$ be the set
\[ \cC = \{ \beta \in \RR^n \mid \mbox{exists } \alpha \in \RR^m\ \text{such
that}\
L(\alpha, \beta) \succ 0\}.\]
Since $L$ is onto and linear, $\cC$ is nonempty and convex. If $0 \in \cC$,
then $L(\alpha, 0) \succ 0$ for some $\alpha$, which is a contradiction.
Therefore, $0 \not\in \cC$, which implies that there exists $x \in \RR^n
\setminus \{0\}$ such that
$\langle x, \beta \rangle \geq 0$ for all $\beta \in \cC$.  Therefore, for each
positive-definite matrix, which, since $L$ is onto, must be of the form
$L(\alpha, \beta) \succ 0$,
\[ \langle L(0, x), L(\alpha, \beta) \rangle = \langle x, \beta \rangle \geq
0,\]
which implies that $L(0, x) \succeq 0$.
Similarly, there exists $\alpha \in \RR^m
\setminus \{0\}$ such that $L(\alpha,0) \succeq 0$.
\end{proof}

We now use Lemma \ref{lem:pSd} to construct positive linear functionals.

\begin{lemma}
\label{lem:goodSepFun}
Let
$\rC \subset \RR^{1
\times \ell}\axs$
be a finite right chip space and let $I \subset \RR\axs^{1 \times \ell}$ be
a real left
module
generated by polynomials in $\RR\axs_1\rC$. There
exists a positive linear functional $L$ on
$\rC^*\RR\axs_2\rC$ such that
the following hold:
\begin{enumerate}
 \item $L(a^*a) > 0$ for each $a \in \RR\axs_1\rC \setminus I$
 \item $L(\iota) = 0$ for each $\iota \in (\RR^{\ell \times 1} I + I^*
\RR^{1 \times \ell} ) \cap \rC^* \RR\axs_2 \rC$.
\end{enumerate}
\end{lemma}

\begin{proof}
Let $\RR\axs_1 \rC = I \oplus T$ for some space $T$.
Let $\tau$ be a vector of length $\mu$ whose entries form a basis for $T$.
Let $\cZ \subset \mathbb{S}^{\mu}$ be defined by
\[
\cZ :=
\left\{
Z \mid
\tau^*
Z
\tau
 \in \RR^{\ell \times 1}I + I^*\RR^{1 \times \ell}
\right\}.\]
Since $I$ is real, the space $\cZ$ contains no $Z \neq 0$ with
$Z \succeq
0$.
By Lemma \ref{lem:pSd} there exists a positive-definite matrix $C \in
\widehat{\cZ}^{\bot}$.
By Lemma \ref{lem:Hankel}, there exists a positive linear functional
$L$ on $\rC^*\RR\axs_2 \rC$ which gives the result.
\end{proof}

\begin{lemma}
 \label{lem:main}
 Let  $\rC \subset
\RR^{1 \times \ell}\axs$
be a finite right chip space and let
$I
\subsetneq \RR^{1 \times \ell}\axs$ be a real
left module generated by polynomials in $\RR\axs_1 \rC$.
Let $n = \dim(\rC) - \dim(I \cap \rC)$.
There exists $(X,v) \in V(I)^{(n)}$ such that $p(X)v \neq 0$ if $p \in \rC
\setminus I$.
\end{lemma}

\begin{proof}

First, $I \neq \RR^{1 \times \ell}\axs$ implies that $n > 0$.
Let $L$ be a linear functional with the properties described by
Lemma
\ref{lem:goodSepFun}.  We see that $L$ restricts
to a functional on $\rC^*\RR\axs_1\rC$. By Corollary \ref{cor:GNS} we produce a
tuple $(X,v) \in (\RR^{n \times n})^g \times \RR^n$, for some $n \in \NN$, such
that $v^*p(X)v = L(p)$ for each $p \in \rC^*\RR\axs_1
\rC$, and such that
\[
 \RR^n = \{c(X)v \mid c \in \rC \}.
\]
Therefore if $a \in \rC$,
\[
 \|a(X)v \|^2 = v^*a(X)^*a(X)v = L(a^*a)
\]
which is nonzero if and only if $a \not\in I$.
Further, if $\iota \in I \cap \RR\axs_1 \rC$, then $L(c^*\iota) = 0$ for each $c
\in \rC$ by
Proposition \ref{prop:flatExtRC}.  Since $\iota(X)v \in \RR^n$, there
exists some $c \in \rC$ such that $c(X)v = \iota(X)v$ and so
\[\|\iota(X)v\|^2 = v^*c^*(X)\iota(X)v = L(c^*\iota) = 0.\]
 Since $I$ is generated by its elements in $\RR\axs_1 \rC$, this implies that
$(X,v) \in V(I)$.
\end{proof}

\subsection{The Matrix Non-Commutative Left Nullstellensatz}
\label{sec:mainResults}

\begin{prop}
   \label{thm:lnss}
If $I \subset \RR^{1 \times \ell}\axs$ is a finitely-generated left module,
then $\rr{I} = \sqrt{I}$.
\end{prop}

\begin{proof}
Let $p \in \RR^{1 \times \ell}\axs$.  Choose $d$ sufficiently large so that
$\deg(p) \leq d$ and so that $I$ is generated by polynomials with degree
bounded by $d$. Let $\rC = \RR^{1 \times \ell}\axs_{d}$.  Then Lemma
\ref{lem:main} implies the existence of a tuple $(X,v) \in V(I)$ such that
$p(X)v \neq 0$.
\end{proof}

We now prove Theorem \ref{thm:mainFromNotes}.

\begin{proof}[Proof of Theorem \ref{thm:mainFromNotes}]
 Note that $p_i(X)v = 0$ means each row of $p_i(X)v = 0$, i.e. $e_k^*p_i(X)v =
0$ for each $e_k \in \RR^{1 \times \nu_i}$.  Therefore
\[ V(I) = V\left( \sum_{i=1}^k
\RR^{1 \times \nu_i} \axs p_i \right).\]
The first part of the result follows from Proposition \ref{thm:lnss}.

Next, if $q$ is an element
of the real left module (\ref{eq:checkIReal}), then
\[q = \sum_i^{\finite} \sum_{j=1}^k a_{ij} b_{ij} p_j\]
for some $a_{ij} \in \RR^{\nu \times 1}$ and $b_{ij} \in \RR^{1 \times
\ell}\axs$.
Therefore,
\[q = \left( \sum_i^{\finite} a_{ij} b_{ij} \right) p_j. \]
\end{proof}

\section{Extension to $\mathbb{C}$ and $\mathbb{H}$}
\label{sec:CandH}

We now show how to extend the main results of the paper to the case where the
polynomials have complex or quaternion coefficients.

There are well-known injective homomorphisms of $\CC$ and $\HH$ into $\RR^{2
\times 2}$ and $\RR^{4 \times 4}$ respectively.  It therefore makes sense to
think of matrices of NC polynomials in with coefficients in $\CC$ or $\HH$ as
matrices of NC polynomials with coefficients in $\RR$.

Let $\CC\axs$ be the space of NC polynomials with coefficients in $\CC$.  Here,
the involution $\ast$ acts on $\CC$ by conjugation so that for each $p \in
\CC\axs$ we have $p^*(X) = p(X)^*$, where $\ast$ on complex matrices denotes the
conjugate transpose.

Let $\HH\axs$ denote the space of NC polynomials over $\HH$.  Here the letters
$x_i$ and $x_j^*$ do not commute with the non-real elements of $\HH$ because, in
general, if $a \in \HH \setminus \RR$ and $X \in \HH^{n \times n}$, then $aX
\neq Xa$.

We also define a space $\HH_c \azs$ to be the space of NC polynomials over
$\HH$ where the letters $z_i$ and $z_j^*$ commute with $\HH$.

Over these spaces, there are precise analogs of left module, zero set, radical
left module, and real left module, which will be denoted as they were in the
real case.

\subsection{Quaternion Case}

The most general case $\HH\axs$.
We now present some simple results for that case, noting that the complex valued
case is similar (but slightly easier).

Let $\psi: \HH \rightarrow \RR^{1 \times 4}$
be the $\RR$-linear bijection defined by
\[
 \psi(a + b i + cj + d k) = (a, b, c, d).
\]
We further extend $\psi$ to $\HH^{\nu \times \ell}$ by coordinates.
Further, we can extend $\psi$ to map $\HH^{\nu \times \ell}_c\azs$
into $\RR^{4\nu \times 4\ell}\axs$  by applying $\psi$ to
coefficients and replacing $\azs$ with $\axs$.
This extension of $\psi$ is $\RR\axs$-linear.

\begin{prop}
\label{prop:realProj}
 If $I \subset \HH_c^{1 \times \ell}\axs$ is a left module, then $J = \psi(I)
\subset \RR^{1 \times 4 \ell}\axs$ is also a left module.  Further, $I$ is real
if and only if $J$ is.
\end{prop}

\begin{proof}
 If $a, b \in I$, then clearly $\psi(a) + \psi(b) = \psi(a + b) \in J$ since
$\psi$ is linear. Further, if $c \in \RR\axs$, then $c\psi(a) = \psi(ca) \in
J$.  Therefore $J$ is a left module.

It is an easy exercise to show that $I$ being real implies that $J$ is real.
Conversely, suppose
\[
 \sum_i^{\rm finite} p_i^*p_i  = \sum_j^{\rm finite} q_j^*r_j + \sum_j r_j^*q_j
\]
where each $p_i, q_j, r_j \in \HH_c^{1 \times \ell}\azs$, and $r_j \in I$.
Let $\phi: \HH \rightarrow \RR^{4 \times 4}$ be the injective homomorphism
\[
 \phi(a + bi + cj + d k) = \begin{pmatrix}
                 a&b&c&d\\-b&a&-d&c\\-c&d&a&-b\\-d&c&b&a
                \end{pmatrix}
                =
                \begin{pmatrix}
                \psi(a + bi + cj + dk)\\
                \psi(i(a + bi + cj + dk)) \\
                \psi(j(a + bi + cj + dk)) \\
                \psi(k(a + bi + cj + dk))
                \end{pmatrix}
\]
and extend $\phi$ to $\HH_c^{1 \times \ell}\azs$ by coordinates.
Then
\[
  \sum_i^{\rm finite} \phi(p_i)^*\phi(p_i)  = \sum_j^{\rm finite}
\phi(q_j)^*\phi(r_j) + \sum_j \phi(r_j)^*\phi(q_j).
\]
The rows of each $\phi(r_j)$ are $\psi(r_j), \psi(ir_j), \psi(jr_j), \psi(kr_j)
\in J$. Therefore, if $J$ is real, then each row of each $\phi(p_i)$ is in $J$.
In particular, $\psi(p_i)$ is a row of $\phi(p_i)$, so each $p_i \in I$.
\end{proof}

If $x = (x_1, \ldots, x_g)$ and $z = (z_1, \ldots, z_{4g})$, define a map
$\varphi: \HH\axs \rightarrow \HH_c\azs$ by
\[
\varphi(p) = p(z_1 + i z_2 + jz_3 + kz_4, \ldots, z_{4g-3} + i z_{4g-2} +
jz_{4g-1} + k z_{4g}).
\]
Clearly $\varphi$ is an injective homomorphism.  Further, we see that $\varphi$
is actually surjective and has inverse given by the maps
\begin{align}
\notag
z_{4m - 3} &\mapsto \frac{1}{4}(x_m - i x_m i - j x_m j - k x_m k)\\
\notag
z_{4m - 2} &\mapsto \frac{1}{4}(-ix_m - x_m i - k x_m j + j x_m k)\\
\notag
z_{4m - 1} &\mapsto \frac{1}{4}(-jx_m + k x_m i - x_m j -i x_m k)\\
\notag
z_{4m} &\mapsto \frac{1}{4}(-kx_m -jx_m i +i x_m j - x_m k)
\end{align}

\begin{prop}
\label{prop:realProj2}
Let $I \subset \HH\axs$ be a subset and let $J = \varphi(I)$.  Then $I$ is a
left module if and only if $J$ is.  Further, if $I$ and $J$ are left modules,
$I$ is real if and only if $J$ is.
\end{prop}

\begin{proof}
This follows easily from $\varphi$ being a bijective homomorphism between
$\HH\axs$ and $\HH_c \azs$.
\end{proof}

\subsection{Extension of the Left Nullstellensatz to $\CC$ and $\HH$}

\begin{corollary}
\label{cor:extensionCH}
Let $\mathbb{F} = \RR, \CC$ or $\HH$.
Let $p_1, \ldots, p_k$ be such that each
$p_i \in \mathbb{F}^{\nu_i \times \ell}\axs$ for some $\nu_i \in \NN$.
Define
\[
 J_\nu := \setmult{\mathbb{F}^{\nu \times 1}}{\rr{\sum_{i=1}^k
\setmult{\mathbb{F}^{1
\times \nu_i}\axs}{ p_i}}}
\]
for $\nu \in \NN$.
Let
$q \in \mathbb{F}^{\nu \times \ell}\axs$.
Then $q(X)v = 0$ for all $(X,v) \in  \bigcup_{n \in \NN} (\mathbb{F}^{n \times 
n})^g
\times
\mathbb{F}^{\ell n}$ such that
$p_1(X)v, \ldots, p_k(X)v = 0$ if and only if  $q \in J_\nu$.
\end{corollary}

\begin{proof}
The $\RR$ case is Theorem \ref{thm:mainFromNotes}.
 As we saw in in that case, this result boils down to showing $\sqrt{I} =
\sqrt[\real]{I}$ for any finitely-generated left module $I \subset \mathbb{F}^{1
\times \ell}\axs$.  We will do the $\mathbb{F} = \HH$ case here; the $\CC$ case 
is
similar but easier.

Let $q \not\in \sqrt[\real]{I}$.
It suffices to show that $q \not\in \sqrt{I}$.
We see that $\psi \circ \varphi(q) \not\in \psi \circ \varphi 
(\sqrt[\real]{I})$.
Propositions \ref{prop:realProj} and \ref{prop:realProj2} imply that $ \psi
\circ \varphi (\sqrt[\real]{I})$ is real, so
\[
\psi \circ \varphi(I) \subset \sqrt[\real]{\psi \circ \varphi (I)} \subset
\psi \circ \varphi (\sqrt[\real]{I})
\]
which implies that $\psi \circ \varphi (q) \not\in \sqrt[\real]{\psi \circ
\varphi (I)}$.
Proposition \ref{thm:lnss} implies that there exists a real tuple $(Z, v) \in
V(\psi(I))$ such that $\psi \circ \varphi (q) (Z)v \neq 0$.  Express 
$\varphi(q)$ as
\[
\varphi(q) = (q_1 + i q_2 + j q_3 + k q_4, \ldots, q_{4 \ell - 3} + i q_{4 \ell 
-2 } + j
q_{4 \ell -1} + k q_{4 \ell}).
\]
Then,
\[
(\psi (q_1)(Z), \ldots, \psi (q_{4\ell})(Z))
\begin{pmatrix}
v_1\\ v_2 \\ \vdots \\ v_{4 \ell}
\end{pmatrix} = \sum_{i=1}^{4 \ell} \psi (q_i)(Z) v_i \neq 0.
\]
Let $w \in \HH^{\ell}$ be
\[
w=
\begin{pmatrix}
v_1 - i v_2 - j v_3 - k v_4\\ \vdots \\ v_{4\ell -3} - i v_{4 \ell - 2} - j v_{4
\ell - 1} - k v_{4 \ell}
\end{pmatrix}
\]
and let $X \in (\HH^{n \times n})^g$ be
\[
X = (Z_1 + i Z_2 + j Z_3 + k Z_4, \ldots, Z_{4g - 3} + i Z_{4g - 2} + j Z_{4g-1}
+ k Z_{4g}).
\]
If $r \in \HH^{1 \times \ell}\axs$, then
\begin{align}
\notag
r(X) &=  \varphi(r)(Z),
\end{align}
and further, it is easy to show that
\[
\operatorname{Re}(\varphi(r)(Z) w) = \psi \circ \varphi (r)(Z) v.
\]
Therefore
\[
\operatorname{Re}(q(X) w)
= \psi \circ \varphi (q)(X)v \neq 0.
\]
Also, if $p \in \sqrt[\real]{I}$, then for each $a \in \HH$,
\begin{align}
\notag
\operatorname{Re}(ap(X)w) &= \psi \circ \varphi (ap)(X)v = 0,\end{align}
which implies that $p(X)w = 0$.  Therefore $(X, w) \in V(I)$, which shows that
$q \not\in \sqrt{I}$.
\end{proof}

Note that a Corollary of this result is an extension of \cite[Theorem 1.6]
{chmn} to $\HH$.

\section{The Real Radical of a Left Module in \texorpdfstring{$\RR^{1 \times
\ell}\axs$}{[R 1 times ell x xs]}}
\label{sec:realRadStuff}

We now use the results of Sections \ref{sec:COrders} and
\ref{sec:DoubleCOrders} to prove a strong result, Theorem
\ref{thm:reducedRealTest}, about the real radical of a finitely-generated left
ideal.  This result is both a generalization of and an improvement upon
\cite[Corollary 2.6]{chmn}.  We prepare for the proof of Theorem
\ref{thm:reducedRealTest} with several lemmas.

\begin{lemma}
\label{lem:leadOfSquare}
  Let $\prec_0$ be a degree order on $\axs$ such if  $ a = a_1a_2$, $b =
b_1b_2$, where $|a_1| = |a_2| = |b_1|
= |b_2| = d$ for some degree $d$, then $a \prec_0 b$ if one of the following
holds:
\begin{enumerate}
 \item $a_2 \prec_0 b_2$, or
 \item $a_2 = b_2$ and $a_1 \prec_0 b_1$.
\end{enumerate}
Let $\rC \subset \RR^{1 \times \ell}\axs$ be a right chip space,
and let $\prec_{\rC}$ be a $\rC$-order induced by $\prec_0$, and let
$\prec_{\rC \times \rC}$ be a double $\rC$-order induced by $\prec_{\rC}$.
Let $p \in \RR\axs \rC \setminus \rC$.  Then the leading monomial of
$p^*p$ is $\lead{p}^*\lead{p}$, where $\lead{p}$ is the leading monomial of
$p$. Further, $\lead{p}^*\lead{p}$ has a positive coefficient in $p^*p$.
\end{lemma}

\begin{proof}
 Let $p$ be decomposed as
\[p = \sum_{w \in \axs}\sum_{m \in \RR\axs \rC \setminus \rC} A(w,m) wm +
\bar{p},\]
where $A(w,m) \in \RR$ and $\bar{p} \in \rC$.
Since $p \in \RR\axs \rC \setminus \rC$ by assumption, at least some of the
$A(w,m)$ are nonzero.
Let $d$ be the maximum length of any $w$ such that $A(w,m) \neq 0$ for any $m$.

For any $m_1,m_2 \in \RR\axs \rC
\setminus \rC$ and $w_1, w_2 \in \axs$ the representation $(w_1 m_1)^*(w_2 m_2)
= m_1^*(w_1^*w_2)m_2$ is the unique representation given by Lemma
\ref{lem:repNEW2} with left and right factors
being in $\RR\axs_1 \rC \setminus \rC$.

If $c \in \rC$, $m \in \RR\axs \rC \setminus \rC$, and $|w| \leq d$, then
either $m^*w^*c$ and $c^*wm$ are in $\rC^*\RR\axs_1 \rC$, or $m^*w^*c =
m^*w_2^*(w_1^*c)$ and $c^*wm = (w_1^*c)^*w_2m$, where $w = w_1w_2$ and $w_1^*c
\in \RR\axs \rC \setminus \rC$, with $|w_1| < |w|$.
Therefore, the terms of
$p^*p$ in $\rC^*\RR\axs \rC \setminus \rC^*\RR\axs_1 \rC$ whose representation
given by Lemma \ref{lem:repNEW2} has a middle word of maximum length
are those of the form $m_1^*w_1^*w_2m_2$ with $m_1, m_2 \in \RR\axs_1 \rC 
\setminus \rC$ and $|w_1| = |w_2| = d$.
Each such term of $p$ is uniquely a product of $(w_1m_1)^*$ and $w_2m_2$ since
$d$
is
the maximum length of the leftmost word $w_1$ and $w_2$ in $p$.

In $p^*p$, the monomial
$m_1^*w_1^*w_2m_2$ has coefficient $A(w_1,m_1)A(w_2,m_2)$, and hence is
nonzero if and only if both $A(w_1,m_1)$ and $A(w_2,m_2)$ are nonzero.  Let
$w_*m_* = \lead{p}$ be
the leading monomial of $p$.
If $A(w_1,m_1), A(w_2,m_2) \neq 0$, we have $w_1m_1,w_2m_2 \preceq_{\rC}
w_*m_*$.  Then clearly $m_1^*w_1^*w_2m_2 \prec_{\rC
\times \rC} m_*^*w_*^*w_*m_*$.
Therefore, the leading monomial of $p^*p$ is $\lead{p}^*\lead{p}$.
Further, the coefficient of $\lead{p}^*\lead{p}$ is $A(m_*w_*)^2 > 0$.
\end{proof}

\begin{lemma}
\label{lem:sosInIC}
 Let $\rC \subset \RR^{1 \times \ell}\axs$ be a finite right chip space.
Let $I \subset \RR^{1 \times \ell}\axs$ be a left
module
generated by polynomials in $\RR\axs_1 \rC$.
If
\begin{equation}
\label{eq:hereASquare}
 \sum_i^{\finite} p_i^*p_i \in \RR^{\ell \times 1} I + I^* \RR^{1 \times
\ell},
\end{equation}
for some $p_i \in \RR^{1 \times \ell}\axs$,
then each $p_i \in I + \rC$.
\end{lemma}

\begin{proof}
Suppose (\ref{eq:hereASquare}) holds.
Let $\Theta$ be the space
\[\Theta = \sum_{j \not\in \Gamma(\rC)} e_j \otimes \RR\axs,\]
so that, by Lemma \ref{lem:gammaC}, $\RR^{1 \times \ell}\axs = \RR\axs \rC
\oplus \Theta$.  Let $p_i = c_i + \theta_i$ for each $i$, where $c_i \in
\RR\axs \rC$ and $\theta_i \in \Theta$.  We see that
\[\sum_i^{\finite} p_i^*p_i
= \sum_i^{\finite} c_i^*c_i + \sum_i^{\finite} c_i^*\theta_i
+ \sum_i^{\finite} \theta_i^*c_i + \sum_i^{\finite} \theta_i^*\theta_i.\]
Since (\ref{eq:oplusCTh}) holds and $I \subset \rC$, it must be that
$\sum_i \theta_i^*\theta_i = 0$, which can only occur if each $\theta_i = 0$.
Therefore each $p_i \in \RR\axs \rC$.

Given a polynomial $p_i$, either $p_i \in \rC$ or $p_i \in 
\RR\axs \setminus \rC$ and, by Lemma \ref{lem:leadOfSquare}, $\lead{p_i^*p_i} = 
\lead{p_i}^*\lead{p_i}$.
In the latter case, the leading monomial of
$\sum_i p_i^*p_i$ must
be the
maximal
$\lead{p_i}^*\lead{p_i}$
since the leading monomials of the $p_i^*p_i$, with $p_i \not\in \rC$, cannot
cancel each other
because by Lemma \ref{lem:leadOfSquare} they all have positive coefficients.
Let $p_{i^*}$ such that $\lead{p_{i_*}}$ is maximal.
By Lemma \ref{lem:niceCBasis2}, the leading
monomial $\lead{p_{i_*}}^*\lead{p_{i_*}}$
is of the form $m^*\lead{\iota}$ or $\lead{\iota}^*m$, for some $\iota \in I
\cap \RR\axs_1 \rC \setminus \rC$ and $m \in \RR\axs \rC \setminus \rC$.
Since $\lead{p_{i_*}} \not\in \rC$, decompose $\lead{p_{i_*}}$ by Lemma
\ref{lem:repRRaxsrC} as $u
\bar{p}$, where $u \in \axs$ and $\bar{p} \in \RR\axs_1 \rC \setminus
\rC$.
By Lemma \ref{lem:canFactorMon}, we must have $\bar{p} = \lead{\iota}$.
Let $A$ be the coefficient of $\lead{p_{i_*}}$ in $p_{i_*}$.
Then
\[ \sum_{i \neq i_*} p_i^*p_i + (p_{i_*} - A u
\iota)^*(p_{i_*} - A u \iota) \in  \RR^{\ell \times 1} I + I^* \RR^{1 \times
\ell},\]
and further, $p_{i_*} - A u \iota$ has a smaller leading monomial than $p_{i_*}$
does,
and $p_{i_*} \in I + \rC$ if and only if $p_{i_*} - A u \iota \in I + \rC$.

We repeat this process inductively
to show that each $p_i \in I + \rC$.
\end{proof}

The following theorem is a key result in computing the
real radical of a finitely-generated left module.  This result is a
generalization of \cite[Corollary 2.6]{chmn} to $\RR^{1 \times \ell}\axs$.
Further, the following result gives
is more refined than \cite[Corollary 2.6]{chmn} for verifying whether or not
a left module is real.

\begin{theorem}
\label{thm:reducedRealTest}
 Let $\rC \subset \RR^{1 \times \ell}\axs$ be a
finite right chip space, and let $I \subset \RR^{1 \times \ell}\axs$ be a
left module generated by polynomials in
$\RR\axs_1\rC$.
 Let $(\{\iota_i\}_{i=1}^{\mu},
\{\vartheta_j\}_{j=1}^{\sigma})$ be
a $\rC$-basis for $I$.
Then $I$ is real if and only if whenever
\begin{equation}
 \label{eq:reducedRealTest}
\sum_i^{\finite} p_i^*p_i = \sum_{j=1}^{\mu} (q_j \iota_j + \iota_j^*q_j^*)
+ \sum_{k=1}^{\sigma} (\alpha_k^* \vartheta_k + \vartheta_k^*\alpha_k),
\end{equation}
for some $p_i, q_j \in \rC$ and $\alpha_k \in \RR^{1 \times \ell} \cap \rC$,
then each $p_i \in I$.
\end{theorem}

Note that Theorem \ref{thm:reducedRealTest} implies that to test whether a
left module $I \subset \RR^{1 \times \ell}\axs$ is real, given that $I$ is
generated by polynomials in $\RR\axs_1 \rC$ for some right chip space $\rC$, one
needs only verify whether given some polynomials $p_i \in \rC$ such that
\[
 \sum_i^{\finite} p_i^*p_i \in \RR^{\ell \times 1} I + I^* \RR^{1 \times \ell}
\]
if each $p_i$ must be in $I$.

\begin{proof}
Suppose
\[\sum_i^{\finite} p_i^*p_i \in \RR^{\ell \times 1}I + I^*\RR^{1 \times \ell}.\]
Lemma \ref{lem:sosInIC} implies that each $p_i$ is of the form $p_i = \phi_i +
\psi_i$, where $\phi_i \in I$ and $\psi_i \in \rC$.  Therefore
\[\sum_i^{\finite} \psi_i^*\psi_i \in \RR^{\ell \times 1}I + I^*\RR^{1 \times
\ell},\]
and so $I$ is real if and only if we must have each $\psi_i \in I$.
We get the righthand side of (\ref{eq:reducedRealTest}) from Lemma
\ref{lem:niceCBasis2}
(\ref{it:canRedC}), noting that a sum of squares is necessarily
symmetric.
\end{proof}

Note that Theorem \ref{thm:reducedRealTest} implies degree bounds.  For
example, if $I \subset \RR^{1 \times \ell}\axs$ is generated by polynomials of
degree bounded by $d$, one could use $\rC = \RR^{1 \times \ell}\axs_{d-1}$.
However, Theorem \ref{thm:reducedRealTest} is more refined since in many cases
there exists a smaller right chip space $\rC$ such that $I$ is generated by
polynomials in $\RR\axs_1 \rC$.

\section{Left Gr\"obner Bases}
\label{sect:LGB}

Classically, Gr\"obner bases are used to verify whether a given polynomial $p$
belongs to a given ideal $I$.
We need left Gr\"obner bases for verifying membership in left modules $I
\subset \RR^{1 \times \ell}\axs$; specifically, we will need them for the Real
Radical Algorithm in $\S$ \ref{sub:rralg}.

Fortunately, there is a general theory of one-sided Gr\"obner bases for
one-sided modules
with coherent bases over algebras with ordered multiplicative basis
\cite{Gre00}.
In this section, we give a version of this theory specific to our case and tie
it in with the $\rC$-basis theory previously presented.
 Left Gr\"obner bases
are easily
computable and are used to algorithmically determine membership in a left
module.

A {\bf left admissible order} $\prec$ on $\axs$ is a well order on $\axs$ such
that $a \prec b$ for some $a, b \in \axs$ implies that for each $c \in \axs$ we
have $ca \prec cb$.
Given a left module $I \subset \RR^{1 \times
\ell}\axs$, a subset $\cG \subset I$
is a {\bf left Gr\"obner basis of $I$ with respect to $\prec$}\index{left
Grobner basis} if the left module generated by $\lead{\cG}$ equals the left
module generated by $\lead{I}$.
We say a polynomial $p$ is {\bf monic} if the coefficient of $\lead{p}$ in $p$
is $1$.
We say a left Gr\"obner basis $\cG$ is {\bf reduced} if the following hold:
\begin{enumerate}
\item Every element of $\cG$ is monic.
 \item If $\iota_1, \iota_2 \in \cG$, then $\lead{\iota_1}$ does not divide any
of the terms of $\iota_2$ on the right.
\end{enumerate}

\begin{prop}
 Let $I \subset \RR^{1 \times \ell}\axs$ be a left module and let $\prec$ be a
left admissible order. Then
 \begin{enumerate}
  \item There is a left Gr\"obner basis for $I$ with respect to $\prec$.
  \item There is a unique reduced left Gr\"obner basis for $I$ with respect to
$\prec$.
  \item If $\cG$ is a left Gr\"obner basis for $I$ with respect to $\prec$,
then $\cG$ generates $I$ as a left module.
 \item $\RR^{1 \times \ell} = I \oplus \operatorname{Span} \left( \nonlead{I}
\right)$.
 \end{enumerate}
\end{prop}

\begin{proof}
 See \cite[Propositions 4.2, 4.4]{Gre00}.
\end{proof}

\begin{lemma}
\label{lem:lGBVerMemOG}
Let $I \subset \RR^{\nu \times \ell}\axs$ be a left module and let
$\{\iota_i\}_{i\in \alpha}$ be a left Gr\"obner basis for $I$.
Every element $p \in I$ can be expressed
uniquely as
\begin{equation}
 \label{eq:eltsOfI}
p = \sum_{i}^{\finite} q_i \iota_i,
\end{equation}
for some $q_i \in \RR\axs$.
In particular,
the leading monomial of $p$ is divisible on the
right by the leading monomial of one of the left Gr\"obner basis elements
$\iota_i$.
\end{lemma}

\begin{proof}
Since $\{\iota_i\}_{i\in \alpha}$ generates $I$,
every element $p \in I$ can be expressed as (\ref{eq:eltsOfI}).
Consider the leading monomial of such a $p$.
 If $a$ is a monomial such that $a \prec \lead{\iota_i}$, then for each $r \in
\axs$,
we have $ra \prec r\lead{\iota_i}$.
Therefore the leading monomial of each $q_i \iota_i$ is of the form
$\tilde{q}_i\lead{\iota_i}$, where $\tilde{q}_i \in \axs$ is some monomial
appearing in $q_i$.

Suppose $\tilde{q}_i \lead{\iota_i} = \tilde{q}_j \lead{\iota_j} \neq 0$ for
some $i
\neq j$, and suppose $|\lead{\iota_i}| \leq |\lead{\iota_j}|$.
If $\lead{\iota_i} = E_{a_ib_i} \otimes w_i$ and $\lead{\iota_j} = E_{a_jb_j}
\otimes
w_j$ for some $E_{a_ib_i}, E_{a_jb_j} \in \RR^{\nu \times \ell}$ and $w_i, w_j
\in
\axs$, then $E_{a_ib_i} \otimes \tilde{q}_i w_i = E_{a_jb_j} \otimes \tilde{q}_j
w_j$.
Therefore $a_i = a_j$, $b_i = b_j$, and $\tilde{q}_i w_i = \tilde{q}_j w_j$,
with $|w_i| \leq
|w_j|$.  This implies that $\lead{\iota_i}$ divides $\lead{\iota_j}$ on the
right, which contradicts the properties of the left Gr\"obner basis.
 Therefore the
leading monomials of the nonzero $q_i \iota_i$ are all
distinct, which implies by Lemma \ref{lem:leadOfSumGen}, that either each $q_j =
0$, or the leading
monomial of $p$ is the maximal
nonzero $\tilde{q}_i \lead{\iota_i}$. Further, uniqueness follows from 
linearity and from
Lemma \ref{lem:leadOfSumGen}.
\end{proof}

\subsection{Algorithm for Computing Reduced Left Gr\"obner Bases}

Let $\prec$ be a left monomial order on $\RR^{\nu \times \ell}\axs$.
Let
$I$ be the left module generated by polynomials $\iota_1, \ldots, \iota_{\mu}
\in
\RR^{\nu \times \ell}\axs$.
It is easy to show that inputing $\iota_1, \ldots, \iota_{\mu}$ into the
following algorithm computes a
reduced left
Gr\"obner basis for $I$.

\subsubsection{Reduced Left Gr\"obner Basis Algorithm}
\label{subsub:LGBAlgOG}
\begin{enumerate}
 \item Given: $\cG = \{ \iota_1, \ldots, \iota_{\mu}\}$.
\item\label{step:begin} If $0 \in \cG$, remove it.
Further, perform scalar multiplication so that each element of $\cG$ is monic.
\item
For each $\iota_i, \iota_j \in \cG$,
compare $\lead{\iota_i}$ with the terms of $\iota_j$.
\begin{enumerate}
 \item If $\lead{\iota_i}$ divides a term of $\iota_j$ on the right,
let
$q \in \axs$ and $\xi \in \RR$ be such that $\xi q\lead{\iota_i}$ is a
term in $\iota_j$. Replace $\iota_j$ with $\iota_j - \xi q \iota_i$.
Repeat (\ref{step:begin}).
\item If $\lead{\iota_i}$ does not divide any terms of any $\iota_j$
on the right for any $i
\neq j$,
stop and output $\cG$.
\end{enumerate}
\end{enumerate}

\subsection{Reduced Left Gr\"obner Bases are
\texorpdfstring{$\rC$}{[C]}-Bases}

For a well-chosen $\rC$ and $\prec_{\rC}$, a reduced left Gr\"obner basis
is actually a $\rC$-basis.

\begin{prop}
\label{prop:LGBisCBasis}
 Let $I \subset \RR^{1 \times \ell} \axs$ be a finitely-generated left module
with reduced left Gr\"obner basis $\iota_1, \ldots, \iota_{\mu}$ according to
some
left monomial order $\prec$ on $\RR^{1 \times \ell}\axs$.
Let $\rC$ be the right chip space defined by
\[\rC:= \operatorname{Span}(\{m \in \RR^{1 \times \ell}\axs \mid m \mbox{
proper right chip of a term of some } \iota_i \}),\]
and let $\prec_{\rC}$ be a $\rC$-order induced by $\prec$. Then
$(\{\iota_i\}_{i=1}^{\mu}, \emptyset)$ is a $\rC$-basis.
\end{prop}

Not that this algorithm outputs a basis with at most $\mu$ elements whose
degree is no greater than the inputted elements, that the and that the algorithm
is guaranteed to terminate in finite time.

\begin{proof}
First, $\iota_1, \ldots, \iota_{\mu} \in \RR\axs_1 \rC$ by construction.
Next, each leading monomial $\lead{\iota_i}$ must be in $\RR\axs_1 \rC
\setminus \rC$ since otherwise $\lead{\iota_i} \in \rC$, which implies that it
properly divides a term of some other $\iota_j$.
 By Lemma \ref{lem:lGBVerMemOG}, the left Gr\"obner basis satisfies
the conditions of Lemma \ref{lem:niceCBasis} to be a $\rC$-basis.
\end{proof}

A reduced left Gr\"obner basis is
a nice $\rC$-basis since it has no elements of $\rC$ in it.

\section{The Real Radical Algorithm}
\label{sub:rralg}

In some simple cases, as shown in \cite{CHKMN}, it is easy to verify whether a
left module is real. In general, however, an algorithmic approach is needed.
In \cite{chmn}, a Real Radical Algorithm is presented for computing the real
radical of any finitely-generated left ideal $J \subset \RR\axs$.
We now present a new Real Radical Algorithm
which extends the previous Real Radical Algorithm to finitely-generated left
modules $I \subset \RR^{1 \times
\ell}\axs$.  Further, using right chip spaces,
the new Real Radical Algorithm is much more efficient.

Let
$I \subset {\RR^{1 \times \ell} \axs}$ be the left module
generated by polynomials $\iota_1, \ldots, \iota_{\mu}$.
When $\{\iota_1, \ldots, \iota_{\mu}\}$ is inputted into the
following algorithm, the result is a reduced left Gr\"obner basis for $\rr{I}$.

\subsection{The Real Radical Algorithm}
\label{subsub:rrAlg}
\begin{enumerate}
\item Given: $\cG = \{\iota_1, \ldots, \iota_{\mu}\}$.
\item
Fix a degree lexicographic order on $\RR^{1 \times \ell}\axs$.
 Compute a reduced left Gr\"obner basis
from $\cG$, and set $\cG^{(0)}$ to be the resulting basis.

\item Let $i = 0$.

\item
\label{A1}
If $\cG^{(i)}$ only contains constants, then stop and output $\cG^{(i)}$.

\item
Let $\rC^{(i)}$ be the set of monomials which properly divide a term of
any of the polynomials in $\cG^{(i)}$ on the right.

\item For each $\vartheta_j \in \cG^{(i)}$, let $q_j$ be the polynomial
\[q_j = \sum_{m \in \rC^{(i)}} \alpha_{m,j}^{(i)} m\]
where the $\alpha_{m,j}^{(i)}$ are real-valued variables, and test whether or
not the polynomial
\begin{equation}
 \label{eq:sosInIi}
\sum_{\vartheta_j \in \cG^{(i)}} q_j^*\vartheta_j + \vartheta_j^*q_j
\end{equation}
is a nonzero sum of squares for some values of $\alpha_{m,j}^{(i)} \in \RR$.
See \cite{kp} for more on how to verify if a NC polynomial is
a sum of squares, and see \cite{HOSM} and \cite{CKP} for a computer algebra
package which will do this. Our problem is more complicated than a standard
sums of squares problem since we are dealing with a
polynomial with variable coefficients.  Therefore, we now spell out the details
of a sum of squares algorithm.

\textbf{SOS Algorithm}

\begin{enumerate}
\item Given: (\ref{eq:sosInIi}).
\item For each monomial in (\ref{eq:sosInIi}) which is
not in $(\rC^{(i)})^*
(\rC^{(i)})$, set the coefficient equal to $0$.  Solve the
resulting set of
linear equations in terms of the $\alpha_{m,j}^{(i)}$ and
reduce
(\ref{eq:sosInIi}) to have only terms in
$(\rC^{(i)})^*(\rC^{(i)})$.
 \item
Let $M_i$ be a vector whose entries
are all monomials in $\rC^{(i)}$.
If desired, technically we can pick a smaller vector $M_i$
by eliminating monomials of small degree.
(See \cite{kp}).
\item
Let $\cZ^{(i)}$ be space
\[
\cZ^{(i)} = \{ Z \mbox{ a real symmetric matrix} \mid M_i^* Z M_i = 0 \},
\]
and let $Z_{1}^{(i)}, \ldots, Z_{n^{(i)}}^{(i)}$ be a basis for $\cZ^{(i)}$.
\item For each $m^* \vartheta_j + \vartheta_j^* m$,
let $B_{m,j}^{(i)}$ be a symmetric matrix such that
\[m^* \vartheta_j + \vartheta_j^* m = M_i^* B_{m,j}^{(i)} M_i.\]
\item
In the linear pencil
\begin{equation}
\label{eq:linPenAlg}
  L_i(\alpha^{(i)}, \beta^{(i)}) = \sum_{m \in \rC^{(i)}} \sum_{\vartheta_j \in
\cG^{(i)}}
\alpha_{m,j}^{(i)} B_{m,j}^{(i)} + \sum_{k=1}^{n^{(i)}} \zeta_k^{(i)}
Z_{k}^{(i)},
\end{equation}
if there are any diagonal entries which are $0$, set all
entries in the corresponding row and column
to be $0$.  Use the resulting
linear equations to reduce the problem,
and delete the $0$ row and column from $L_i$.
Also delete the entry in $M_i$
with the same index as the deleted row and column.
Repeat this step until there are no diagonal entries in $L_i$ equal to $0$.
\item If we eventually get $L_i = 0$, stop and output
that there is no nonzero sum of squares.
\item Solve the linear matrix inequality
\[
L_i(\alpha^{(i)}, \beta^{(i)}) \succeq 0
\]
to see if there is a nonzero solution $(\alpha^{(i)}, \beta^{(i)})$.
\begin{enumerate}
\item If there is not, stop and output
that there is no nonzero sum of squares.
\item Otherwise, output the vector of polynomials
\[
 \sqrt{L(\alpha^{(i)}, \beta^{(i)})} M_i.
\]
\end{enumerate}
\end{enumerate}
\item If there is no nonzero sum of squares, stop and output $\cG^{(i)}$.
\item Otherwise, let $\phi_1,
\ldots, \phi_{r_i}$ be the entries of the outputted vector of polynomials.
Compute a reduced left Gr\"obner basis for the set $\cG^{(i)} \cup \{\phi_1,
\ldots, \phi_{r_i}\}$ and let $\cG^{(i+1)}$ be the resulting set.
Go to (\ref{A1}).
\end{enumerate}

\subsubsection{Properities of the Real Radical Algorithm}
\label{subsub:propsRrAlg}

We now prove Theorem \ref{thm:algorStops}, which presents
some appealing properties of the Real Algorithm
presented in $\S \ref{subsub:rrAlg}$.

\begin{proof}[Proof of Theorem \ref{thm:algorStops}]
 First, if $\deg(\iota_1), \ldots, \deg(\iota_{\mu}) \leq d$, since
we are using a degree lexicographic order to compute the reduced left Gr\"obner
basis, it is clear from the left Gr\"obner basis algorithm that the outputted
left Gr\"obner
basis also consists of polynomials with degree bounded by $d$.
Next, for the set $\rC^{(i)}$, if $\cG^{(i)}$ consists of polynomials of degree
bounded by $d$, then the only monomials which properly divide a monomial of
degree bounded by $d$ on the right must have degree less than $d$.  Therefore
the set $\rC^{(i)}$ has monomials of length at most $d-1$.
In particular, at each iteration, $\cG^{(i+1)}$ is a reduced left Gr\"obner
basis generated by $\cG^{(i)}$, which has polynomials of degree bounded by $d$,
and some
polynomials in the span of $\rC^{(i)}$.  Therefore, at each step, $\cG^{(i)}$
always consists of polynomials of degree bounded by $d$.
Finally,
the polynomial (\ref{eq:sosInIi}) is a sum of polynomials with degree bounded
by $2d-1$.  This all verifies the degree bounds in (\ref{degreed}).

Second, at each step we are adding polynomials in the span of
$\rC^{(i)}$ to
$\cG^{(i)}$ and then computing a reduced left Gr\"obner basis.  Further, each
monomial in $\rC^{(i)}$ is not divisible on the right by the leading monomial of
an element of $\cG^{(i)}$, so $\operatorname{Span} (\rC^{(i)}) \cap \RR\axs
\cG^{(i)} = \{0\}$.  Therefore
the left module
generated by $\cG^{(i+1)}$ must properly contain the left module generated by
$\cG^{(i)}$.
Since $\RR^{1 \times \ell}\axs_d$ is finite dimensional, this process must stop
eventually.  This proves (\ref{stops}).

Finally, at each step, consider finding a nonzero sum of squares of the form
(\ref{eq:sosInIi}).  Assume inductively that each $\cG^{(i)} \subset \rr{I}$.
This is true at the outset since $\cG^{(0)} \subset I \subset \rr{I}$. If
the SOS algorithm finds such a sum of squares,
then it is equal to
\begin{align}
\notag
M_i^* L_i(\alpha,\beta) M_i &= \left( \sqrt{L_i(\alpha,\beta)}M_i \right)^*
\left( \sqrt{L_i(\alpha,\beta)}M_i\right)
\\
\notag
&\in \RR^{\ell \times 1}\rr{I} +
\rr{I}^*\RR^{1 \times \ell}.
\end{align}
This implies that the outputted vector of polynomials
from the real radical algorithm are polynomials in $\rr{I}$, which gives
$\cG^{(i+1)} \subset \rr{I}$.

If the SOS Algorithm outputs that there is no sum of squares, by Proposition
\ref{prop:LGBisCBasis}, $(\cG^{(i)}, \emptyset)$ is a $\rC$-basis,
so by Theorem \ref{thm:reducedRealTest}, this is enough to show that the left
module generated by $\cG^{(i)}$ is real.  Since $\cG^{(i)} \subset \rr{I}$, this
implies that $\cG^{(i)}$ is a reduced left Gr\"obner basis for $\rr{I}$.
\end{proof}

\section{Acknowledgments}

Author was partly supported by J.W.\ Helton's National Science
Foundation grant DMS 1201498.  Thanks to J.W.\ Helton and Igor Klep for their 
comments and advice.

\newpage

\printindex

\tableofcontents

\end{document}